\documentclass[article]{siamart171218}

\usepackage{graphicx}
\usepackage{amssymb,amsmath,amsfonts}
\usepackage{epsf}
\usepackage{wrapfig}
\usepackage[normalem]{ulem}
\usepackage[mathscr]{euscript}
\usepackage{mathtools}
\usepackage{ifmtarg}
\usepackage[utf8]{inputenc}
\usepackage{fancyvrb}
\newcommand{\ba}{\begin{eqnarray}} 
\newcommand{\ea}{\end{eqnarray}}
\newcommand{\be}{\begin{equation}} 
\newcommand{\ee}{\end{equation}} 
 
\newcommand{\nn}{\nonumber}

\newcommand{\lr}[1]{\left\langle #1 \right\rangle}
\newcommand{\intl}{\int\limits}

\newcommand{\liml}{\lim\limits}
\newcommand{\tbar}{\overline{T}} 
\newtheorem{remark}{Remark}

\newcommand{\R}{\mathbb{R}}

\renewcommand{\matrix}[2]{ \left(\begin{array}{#1} #2 \end{array}\right)}
  
\newcommand{\E}{\mathbb{E}}

\newcommand{\spi}{S\&P}

\newcommand{\mbf}{\mathbf{f}}
\newcommand{\mbv}{\mathbf{v}}
\newcommand{\mbx}{\mathbf{x}}

\def\given{\:|\:}
\usepackage{color}
\newcommand{\rot}[1]{{\color{black} #1}}
\newcommand{\blau}[1]{{\color{black} #1}}
\newcommand{\bi}[1]{Fig.~\ref{fig:#1}}
\newcommand{\e}[1]{eq.~(\ref{eq:#1})}
\definecolor{darkred}{RGB}{150, 0, 0}

\newcommand{\rhoss}{\rho_\text{ss}}

\DefineVerbatimEnvironment{Latex}{Verbatim}{numbers=left,numbersep=2mm}

\makeatletter

  {\list{}{\leftmargin=#1\rightmargin=#1}\item[]}%
  {\endlist}
    \newcommand\contFrac{\@ifstar{\@contFracStar}{\@contFracNoStar}}

    \def\singleContFrac#1#2{%
        \begin{array}{@{}c@{}}%
            \multicolumn{1}{c|}{#1}%
            \\%
            \hline%
            \multicolumn{1}{|c}{#2}%
        \end{array}%
    }

    \def\@contFracNoStar#1{%
        \mathchoice{
            \@contFracNoStarDisplay@#1//\@nil%
        }{
            \@contFracNoStarInline@#1//\@nil%
        }{
            \@contFracNoStarInline@#1//\@nil%
        }{
            \@contFracNoStarInline@#1//\@nil%
        }%
    }

    \def\@contFracNoStarDisplay@#1//#2\@nil{%
        \@ifmtarg{#2}{%
            #1%
        }{%
            #1+\cfrac{1}{\@contFracNoStarDisplay@#2\@nil}%
        }%
    }

        \def\@contFracNoStarInline@#1//#2\@nil{%
            \@ifmtarg{#2}{%
                #1%
            }{%
                #1 \@@contFracNoStarInline@@#2\@nil%
            }%
        }
        \def\@@contFracNoStarInline@@#1//#2\@nil{%
            \@ifmtarg{#2}{%
                + \singleContFrac{1}{#1}%
            }{%
                + \singleContFrac{1}{#1} \@@contFracNoStarInline@@#2\@nil%
            }%
        }

    \def\@contFracStar#1{%
        \mathchoice{
            \@contFracStarDisplay@#1////\@nil%
        }{
            \@contFracStarInline@#1//\@nil%
        }{
            \@contFracStarInline@#1//\@nil%
        }{
            \@contFracStarInline@#1//\@nil%
        }%
    }

    \def\@contFracStarDisplay@#1//#2//#3\@nil{%
        \@ifmtarg{#2}{%
            #1%
        }{%
            #1 + \cfrac{#2}{\@contFracStarDisplay@#3\@nil}%
        }%
    }

        \def\@contFracStarInline@#1//#2\@nil{%
            \@ifmtarg{#2}{%
                #1%
            }{%
                #1 \@@contFracStarInline@@#2\@nil%
            }%
        }
        \def\@@contFracStarInline@@#1//#2//#3\@nil{%
            \@ifmtarg{#3}{%
                - \singleContFrac{#1}{#2}%
            }{%
                - \singleContFrac{#1}{#2} \@@contFracStarInline@@#3\@nil%
            }%
        }
\makeatother

\usepackage{color}

\usepackage{ulem}

\usepackage{blkarray}

\usepackage{lscape}

\usepackage{cancel}

\usepackage{setspace}

\usepackage{caption}
\usepackage{subcaption}

\usepackage{lipsum}
\usepackage{amsfonts}
\usepackage{graphicx}
\usepackage{epstopdf}
\usepackage{algorithmic}
\ifpdf
  \DeclareGraphicsExtensions{.eps,.pdf,.png,.jpg}
\else
  \DeclareGraphicsExtensions{.eps}
\fi

\numberwithin{theorem}{section}

\newcommand{\TheTitle}{A partial differential equation for the \rot{mean--return-time} phase of planar stochastic oscillators} 
\newcommand{\TheAuthors}{A. Cao, B. Lindner, and P. Thomas}

\headers{PDE for the mean--return-time oscillator phase}{\TheAuthors}

\title{{\TheTitle}\thanks{
\funding{B.L. was funded through Deutsche Forschungsgemeinschaft Grant
LI 1046/2-1 and BMBF Grant 01GQ1001A. A.C.  was supported by a grant from the Arthur I. Mendolia '41 Fellowship from the Case Alumni Foundation. A.C. and P.J.T. were funded by National Science Foundation grant DMS-1413770.
}}}

\author{
  Alexander Cao\thanks{Department of Mathematics, Applied Mathematics, and Statistics.
Case Western Reserve University,
Cleveland, Ohio, 44106, USA (\email{axc487@case.edu}).}
  \and
  Benjamin Lindner\thanks{Bernstein Center for Computational Neuroscience and Department of Physics. 
Humboldt University, 10115 Berlin, Germany (\email{benjamin.lindner@physik.hu-berlin.de}).}
  \and
  Peter J.~Thomas\thanks{Department of Mathematics, Applied Mathematics, and Statistics. Case Western Reserve University,
Cleveland, Ohio, 44106, USA (\email{pjthomas@case.edu}).}
}

\usepackage{amsopn}

\ifpdf
\hypersetup{
  pdftitle={\TheTitle},
  pdfauthor={\TheAuthors}
}
\fi

\begin{document}

\maketitle

\begin{abstract}
Stochastic oscillations are ubiquitous in many systems. For deterministic systems, the oscillator's phase has been widely used as an effective one-dimensional description of a higher dimensional dynamics, particularly for driven or coupled  systems.  
Hence, efforts have been made to generalize the phase concept to the stochastic framework. One notion of phase due to Schwabedal and Pikovsky is based on the \rot{mean-return time} (\rot{MRT}) of the oscillator  but has so far been described only in terms of a numerical algorithm. Here we develop the boundary condition by which the partial differential equation for the \rot{MRT} has to be complemented in order to obtain the isochrons  (lines of equal phase)  of a two-dimensional stochastic oscillator, and rigorously establish the existence and uniqueness of the \rot{MRT} isochron function (up to an additive constant). We illustrate the method with a number of examples: the stochastic heteroclinic oscillator (which would not oscillate in the absence of noise);  the isotropic Stuart-Landau oscillator, the Newby-Schwemmer oscillator, and the Stuart-Landau oscillator with  polarized noise. For selected isochrons we confirm by extensive stochastic simulations that the return time  from an isochron to the same isochron (after one complete rotation) is always the mean first-passage time  (irrespective of the initial position on the isochron). Put differently, we verify that Schwabedal and Pikovsky's criterion for an isochron is satisfied. In addition, we discuss how to extend the construction to arbitrary finite dimensions.
Our results will enable development of analytical tools to study and compare  different notions of phase for stochastic oscillators.\footnote{Some results presented here appeared previously in the Master's thesis of the first author.}
\end{abstract}


\section{Introduction}
Stochastic oscillations are an important phenomenon in many areas of science and technology. The activity of certain brain areas \cite{BurkeDePaor2004BICY,PopovychHauptmannTass2006BICY}, repetitive motor activities \cite{GuckenheimerJaveed2018BICY,VanMourikDaffertshoferBeek2006BICY},
the motion of small organelles called hair bundle in our inner ear organs \cite{AshmoreAvanDallosDierkesEtAl2010HearRsch,NeimanDierkesLindnerHanShilnikov2011JMNO}, and the lasing intensity of lasers under certain conditions \cite{Wieczorek2009PRE} are but a few examples that can be modeled by a noisy oscillator.

The dynamic and stochastic mechanisms for generating noisy oscillations are diverse as well. The classical linear example system is the  underdamped harmonic oscillator driven by white Gaussian noise \cite{ThoLin19,WanUhl45}, showing a finite phase coherence because of amplitude and phase fluctuations (similar systems appear in ecological and neural models and are referred  to as quasicycles \cite{LugMcK08,WalBen11}).  

More difficult to analyze are noisy  oscillations occurring in nonlinear stochastic systems. A prominent example is  a  limit cycle  perturbed by noise \cite{FeistelEbeling1978PhysicaA,Tass2007book-phase}.  Qualitatively different is the role of noise in excitable systems \cite{LinGar04}, in heteroclinic oscillators \cite{StoArm99,ThoLin14}, or in certain spatially extended systems \cite{DieJul12,KawSai04} ---  in these cases the driving fluctuations are {\em required} to observe an oscillation (although a noisy one).

In the description of deterministic oscillations and weakly coupled deterministic oscillators a phase description has turned out to be extremely useful in many situations but in particular in neuroscience \cite{Erm96,ErmentroutKopell1986SIAMJAM,HopIzh97,StiefelErmentrout2016JNphys}.
 Hence, more recently, different approaches have been developed to  define a phase also for stochastic oscillations \cite{SchPik13,ThoLin14}. 
 
Schwabedal and Pikovsky (henceforth referred to as \spi) introduced a notion of phase that is based on an algorithmic
numerical procedure involving the \rot{mean---first--return-time; for brevity, we will refer to the \emph{mean--return-time (MRT)}, described below}.  
Alternatively, two of the present authors used the complex phase of the eigenfunction of the backward Kolmogorov operator (the generator of the Markov process) to introduce  an asymptotic phase for stochastic oscillators \rot{\cite{ThoLin14,ThoLin19}}. The two approaches can be regarded as generalizations of the two (deterministically equivalent) notions of phase, one based on return times among a system of  Poincar\'{e} sections, the other on the asymptotic convergence of trajectories \cite{Guc75}. The question of whether one of these phase notions is superior to the other one was subject to some debate \cite{PhysRevLett.115.069401,PhysRevLett.115.069402} but will not be pursued here. 
What will be addressed in this paper is the question of how the algorithmic procedure 
introduced by \spi~in \cite{SchPik13} can be reformulated as the solution of a  partial differential equation. 
We mainly restrict attention to planar Markovian systems that display  stochastic oscillations. This is certainly more limited in scope than the algorithm suggested in \cite{SchPik13}, which has been also applied to non-Markovian examples.
Deriving a partial differential equation to determine the isochrons for the two-dimensional system (the lines that define equal phase), however, opens opportunities of more rigorous mathematical explorations of the \rot{MRT} phase.

\rot{In \cite{SchPik13} \spi~proposed a definition for the phase of a stochastic oscillator in terms of a system of Poincar\'e sections $\{\blau{\ell_{\rm MRT}}(\phi), 0\le \phi\le 2\pi\},$ foliating a domain ${\cal{R}}\subset\R^2$ and possessing  the MRT property.  Intuitively, a section $\blau{\ell_{\rm MRT}}$ satisfies the \blau{MRT} property if for all points $\mbx\in\blau{\ell_{\rm MRT}}$, the mean return time from $\mbx$ back to $\blau{\ell_{\rm MRT}}$ is constant.  
\spi~write:
\begin{quotation}
``For a noisy system we define the isophase surface $J$ as a Poincar\'e surface of section, for which the mean first return time $J\to J$, after performing one full oscillation, is a constant T, which can be interpreted as the average oscillation period. In order for isophases to be well defined, oscillations have to be well defined as well: for example in polar coordinates, the radius variable must never become zero, so that one can reliably recognize each ``oscillation.'' Random processes for which this is not the case should be treated with care."
\end{quotation}

\spi~describe an iterative method for obtaining such a surface from an ensemble of two-dimensional trajectories 
(as obtained, for example, by simulating the Langevin \e{langevin}). Clearly, when naively interpreted, the \blau{MRT corresponds to a} mean first passage time from $\mbx\in\blau{\ell_{\rm MRT}}$ to $\blau{\ell_{\rm MRT}}$, \blau{and is therefore} identically zero; to avoid a trivial definition, \spi~require that the mean be calculated for all trajectories ``after performing one full oscillation.''

While the intended meaning is {intuitively} clear, a 
mathematical formulation of the problem has not previously been given. It is furthermore not clear whether the  method {proposed by \spi~will converge, nor} whether the isochrons (lines {sharing a common} \blau{MRT} phase) are uniquely determined. Last but not least, it is unclear whether the isochrons could also be  obtained by the solution of the partial differential equation (PDE) that governs the \rot{mean--first-passage time} to an absorbing boundary \cite{Gar97}. 

 \spi~apply their iterative method to a wide range of systems (including examples of the kind above, \e{langevin}, but also three-dimensional Markovian systems, \textit{e.g.}~two-dimensional dynamics with colored noise). 
 We show here for the special class of two-dimensional stochastic oscillators driven by white Gaussian noise, \blau{under natural assumptions given below,} that isochrons with the \blau{MRT} property are uniquely defined, and that they correspond to the level curves of the solution of the \blau{mean---first--passage-time} PDE with a periodic-plus-jump condition on an arbitrary simple {curve} connecting the inner and outer boundaries.}

Our paper is organized as follows. 
\rot{In \S\ref{sec:model} we}  
introduce a class of two-dimensional Markovian models  and state our assumptions.
Furthermore, we  develop the \rot{MRT} isochrons as the level sets of a function $T$, unique up to an additive constant; we formulate the boundary conditions complementing the standard  partial differential equation governing this function.  
\rot{We provide two distinct derivations of the main result.  In  \S \ref{ssec:strip} we utilize the unwrapping construction and argue that the MRT function should assume a certain limiting form far upstream of any given transverse boundary.  In \S \ref{sec:proof} we leverage properties of strongly elliptic second order differential operators to give a mathematically precise statement and rigorous proof of the same result. 
Specifically,} through an application of Fredholm theory for strongly elliptic PDEs, we establish  existence and uniqueness of the system of \rot{mean--return-time} isochrons under mild conditions. 
 In addition, \rot{in \S\ref{sec:examples}} we illustrate our results by a number of examples, for which we solve the corresponding PDE with the periodic-plus-jump condition numerically.   
 In \S \ref{sec:ndim} we briefly address the extension of our results to stochastic oscillators in $n>2$ dimensions. We conclude with a short discussion of our results and possible extensions to more general systems.     \rot{In the supplementary materials we discuss details of the numerical method and provide an additional example.}
 
\section{Model and derivation of the main result}
\label{sec:model}
We first present the model of a two-dimensional stochastic oscillator driven by white Gaussian noise, and derive a PDE and its boundary conditions, the solution of which provides the isochrons of the \rot{MRT} phase. 

\subsection{The considered model and the \rot{MRT} phase}
\label{sec:main}
We consider a  planar Langevin system of the form
\begin{figure}[h]
\includegraphics[width=\textwidth]{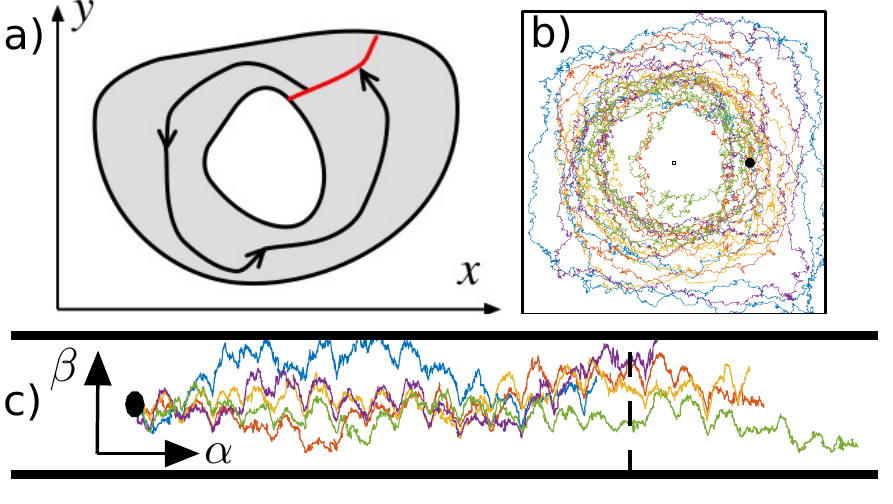}
\caption{\label{fig:sketch_xy} Original ring-like domain with a \rot{counterclockwise} net probability flow and reflecting inner and outer boundaries {(a)}. An isochron (sketched in red) can be defined by the property that the mean return time from any point of this line to itself (after one complete rotation) is given by the overall mean rotation time of the stochastic oscillator.  Five trajectories of the noisy heteroclinic oscillator, \e{hetero}, on the original domain {(b)} and in unwrapped coordinates {(c)}. The common initial condition is marked with a black dot. 
The inner boundary of the annulus (b) is the square $\max(|x|,|y|)=\epsilon=0.05$, which is mapped to the lower edge of the unwrapped domain, $\beta=-1$.  The outer boundary of the annulus is mapped to the upper edge of the unwrapped domain, $\beta=+1$.  The transformation maps point $(x,y)$ to $(\alpha,\beta)$ such that $\tan(\alpha)=y/x$, with $\alpha$ continuous along trajectories. For the quarter wedge $x\ge|y|$ of the annular domain, we set $\beta=-1+2(x-\epsilon)/(\frac{\pi}{2}-\epsilon)$.  In the other quarter wedges of the annular domain, $\beta$ is determined by a similar construction. Vertical dashed line marks one full rotation from the initial point.  Compare \bi{het_jump_results}.
}
\end{figure}
\ba \label{eq:langevin}
\dot{x}&=&f_x(x,y)+g_{x1}(x,y)\xi_1(t)+g_{x2}(x,y)\xi_2(t)\\
\nn
\dot{y}&=&f_y(x,y)+g_{y1}(x,y)\xi_1(t)+g_{y2}(x,y)\xi_2(t)
\ea
where $\xi_{1,2}(t)$ is white Gaussian noise with $\lr{\xi_i(t)\xi_j(t')}=\delta_{ij}\delta(t-t')$;  the multiplicative noise terms are interpreted in the sense of Ito.\footnote{We could equivalently write the system $dx=f_x(x,y)\,dt+g_{x_1}(x,y)\,dW_1(t)+g_{x_2}(x,y)\,dW_2(t)$ and $dy=f_y(x,y)\,dt+g_{y_1}(x,y)\,dW_1(t)+g_{y_2}(x,y)\,dW_2(t)
$, with $dW_1$ and $dW_2$ are the increments of independent standard Wiener processes. See \S \ref{sec:proof} for a rigorous formulation.  For completeness we repeat the equations in Stratonovich form in \S\ref{supp:Stratonovich}.} The domain $\cal{R}$ has the topology of an annulus, formed by the complement of one smooth (piecewise $C^1$) star-shaped domain within another (with common center point), see \bi{sketch_xy}a. 
Outer and inner boundaries will be referred to as $R_{\rm outer}$ and  $R_{\rm inner}$, respectively. 

We assume that the total noise amplitudes do not vanish anywhere in the domain, \textit{i.e.}~
\be\label{eq:posdef1}
g_{x1}^2(x,y)+g_{x2}^2(x,y)>0  \;\;\; \mbox{and}\;\;\; g_{y1}^2(x,y)+g_{y2}^2(x,y)>0 \;\;\; \forall (x,y)\in\cal{R}, 
\ee
that the noise be nonsingular, i.e.\footnote{This condition is required for strong ellipticity, see the proof of existence in \S \ref{sec:proof}.}
\be\label{eq:posdef2}
g_{x1}(x,y)g_{y2}(x,y)-g_{x2}(x,y)g_{y1}(x,y)\not=0\;\;\;\forall (x,y)\in\cal{R},
\ee
and that all drift and diffusion functions are smooth ($C^2$) and bounded; in particular we require that a unique stationary probability density exists and we have a non-vanishing net stationary probability current.  Without loss of generality, \rot{our proof} will assume the current circulates  counterclockwise.  \rot{In \S \ref{sec:examples} we consider specific examples, some of which rotate counterclockwise and others clockwise.}

\subsection{Disambiguation of the return-time problem by mapping to an infinite strip}
\label{ssec:strip}


The region $\mathcal{R}$ is diffeomorphic to the annulus $0<R_-\le r \le R_+<\infty$.
The latter can be mapped to \rot{an angular variable} $\alpha(x,y) \in[0,2\pi)$ and \rot{an amplitude-like variable} $\beta(x,y)\in [-1,1]$, corresponding to a rectangular domain. One possible mapping from the original domain to the rectangle is
\begin{align}
\alpha(x,y)&=\left\{\begin{array}{ll}
\arctan(y/x),&|x|>0\\
\pi/2,&x=0, y>0\\
3\pi/2,&x=0,y<0
\end{array}  \right.\\
\beta(x,y)&=-1+2\frac{\sqrt{x^2+y^2}-R_{\rm inner}(\alpha(x,y))}{|R_{\rm outer}(\alpha(x,y))-R_{\rm inner}(\alpha(x,y))|}
\end{align}

\begin{figure}[h]
	\centerline{\includegraphics[width=0.8\textwidth]{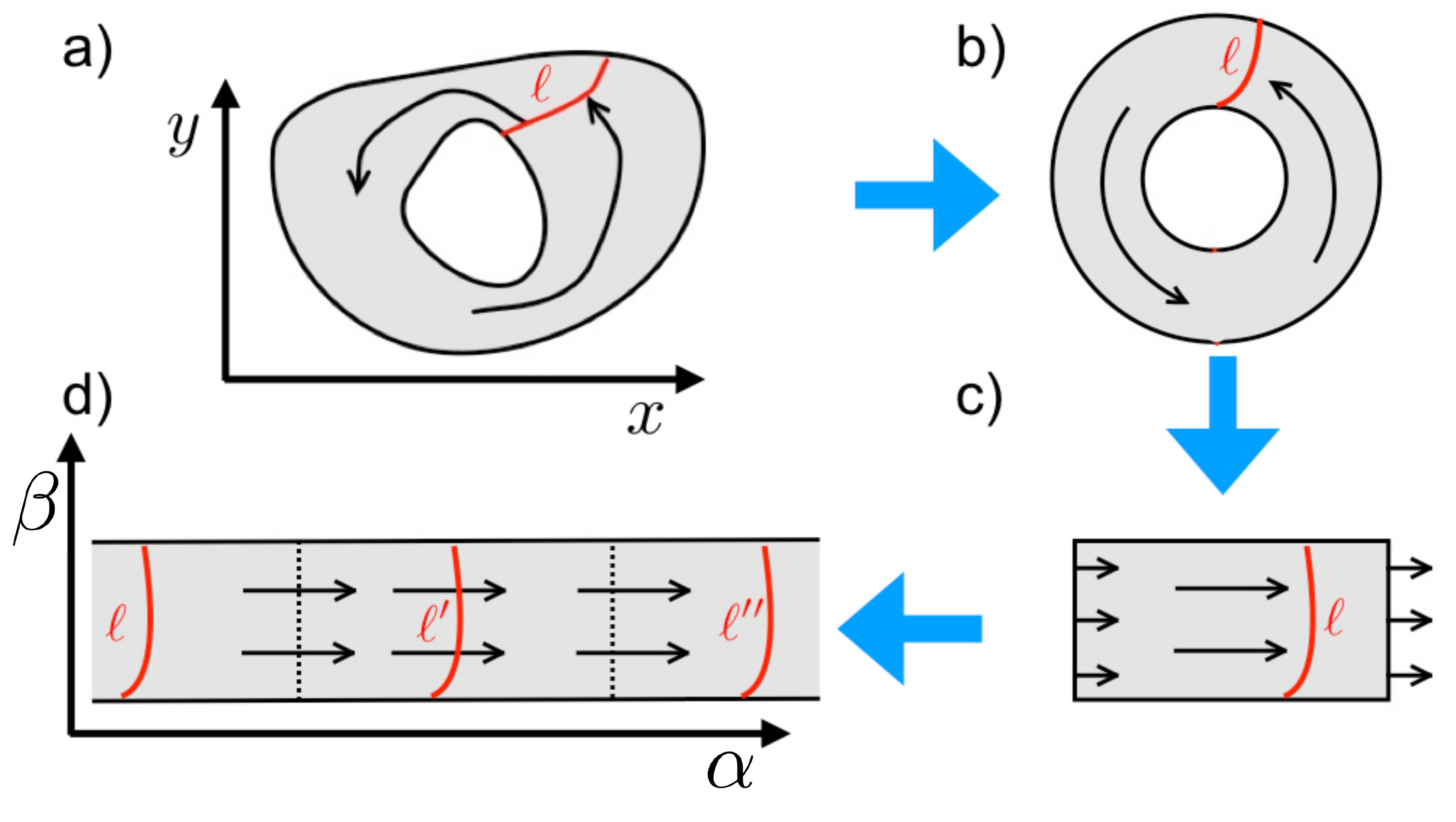}}
\caption{\label{fig:sketch_mapping} The original ring-like domain (a) {(\bi{sketch_xy})} is mapped to a true annulus (b), then to the geometric phase and amplitude with periodic boundary conditions (c) and, finally, to a system with unwrapped phase and amplitude (d). {In (a),} the first return to a simple connection between the two boundaries, $\ell$, after one rotation has been performed, is not well defined: we cannot simply turn the line $\ell$ into an absorbing boundary because then the condition of the performed rotation is not met and if we start on $\ell$ the mean return time will be zero. In the unwrapped phase and amplitude (d), we deal with infinitely many concatenated copies of the system and thus the return-time problem is now well posed: it is given by the first passage problem between two adjacent copies of the line (e.g. from $\ell$ to $\ell'$ or from $\ell'$ to $\ell''$).   
}
\end{figure}

After the change of variables, the SDE attains a new form  
\ba \label{eq:langevin_polar}
\dot{\alpha}&=&f_{\alpha}(\alpha,\beta)+g_{\alpha 1}(\alpha,\beta)\bar{\xi}_1(t)+g_{\alpha 2}(\alpha,\beta)\bar{\xi}_2(t)\\
\nn
\dot{\beta}&=&f_{\beta}(\alpha,\beta)+g_{\beta 1}(\alpha,\beta)\bar{\xi}_1(t)+g_{\beta 2}(\alpha,\beta)\bar{\xi}_2(t)
\ea
with drift and diffusion coefficients uniquely defined by the functions $f_{x,y},g_{x_{1,2}},g_{y_{1,2}}$ \cite{Gar97} 
and again we interpret the equations in the sense of Ito. Note that all coefficients and their derivatives are smooth and periodic in $\alpha$, {\textit{i.e.}~$\forall k\in \mathbb{Z}$ and for  $j=1,2$,}
\ba 
f_{\alpha,\beta}(\alpha,\beta)&=&f_{\alpha,\beta}(\alpha+k\cdot 2\pi,\beta),\\
g_{\alpha,j}(\alpha,\beta)&=&g_{\alpha,j}(\alpha+k\cdot 2\pi,\beta),\\ 
g_{\beta,j}(\alpha,\beta)&=&g_{\beta,j}(\alpha+k\cdot 2\pi,\beta). 
\ea
 The exact form of the coefficients is not needed in the following, but we do require that also in the new variables the total noise intensity does not vanish anywhere {and the noise remain nonsingular} 
\ba \nonumber
g_{\alpha 1}^2(\alpha,\beta)+g_{\alpha 2}^2(\alpha,\beta)&> & 0  \\
g_{\beta 1}^2(\alpha,\beta)+g_{\beta 2}^2(\alpha,\beta)&> & 0 \\ \nonumber
g_{\alpha 1}(\alpha,\beta)g_{\beta 2}(\alpha,\beta)-g_{\alpha 2}(\alpha,\beta)g_{\beta 1}(\alpha,\beta) &\not= & 0 \quad \forall (\alpha,\beta).
\ea
Now instead of using periodic boundary conditions with respect to $\alpha$ on the rectangular domain, we can extend the domain to the left and right by infinitely many copies of it, \textit{i.e.}~we can employ a well-known construction  and use an {\em unwrapped} angle variable and in this way keep track of the rotations; see \bi{sketch_mapping} for a sketch of the different mappings.  

Specifically, the unwrapped phase now solves the problem of the ill-posed return-time problem in the original setup: it is unclear how to impose the condition of `one performed rotation' in the original $x-y$ space because with a stochastic driving we can never exclude that the trajectory might {encircle the inner boundary in the opposite direction from the mean rotation}. In the unwrapped-phase-amplitude domain, the return-time problem has a corresponding and well-posed first-passage-time problem from one copy of the line $\ell$ to the adjacent one on the right, \textit{e.g.}~from $\ell$ to   $\ell'$ or from $\ell'$ to $\ell''$.

The problem of finding the phase implied by the \rot{mean-return} or the mean--first-passage time corresponds in this setting to the problem of finding the {simple connecting curve} $\ell$ such that the mean-first-passage time from any point on $\ell$ to the absorbing boundary at $\ell'$ (and reflecting boundaries at $\beta=\pm 1$ and at $\alpha\to -\infty$) 
is equal to a constant (the mean rotation time of the stochastic oscillator).

\subsection{Forward and backward Kolmogorov operators, probability density and mean first-passage time function on the infinite strip}
\label{ssec:unwrapped}

We can write down the two versions of the Kolmogorov equation, the first one of which is the Fokker-Planck equation 
\be
\partial_t P(\alpha,\beta,t) = {\cal L} P(\alpha,\beta,t)=-\nabla \vec{J}(\alpha,\beta,t).
\label{eq:fpe_theta_r}
\ee
Here, $P(\alpha,\beta,t)$, if started with the initial condition $P(\alpha,\beta,0)=\delta(\alpha-\alpha_0)\delta(\beta-\beta_0)$, is the transition probability density for being at  $\alpha$ and $\beta$ at time $t$, if the system was at $\alpha_0$ and $\beta_0$ at time $t=0$. The vector $\vec{J}=(J_\alpha,J_{\beta})$  denotes the probability current, the $\alpha$ component of which, $J_\alpha(\alpha,\beta,t)$ is related to the mean number of trajectories crossing {a line $\alpha=\text{const}$ at a certain radial coordinate $\beta$.}  
{The probability current is defined in terms of the drift vector $\vec{f}=(f_\alpha,f_{\beta})^\intercal$ and diffusion matrix $\mathcal{G}=gg^\intercal$ by
\be
\vec{J}=\vec{f}P-\frac12
\left[
\begin{array}{cc}
\partial_\alpha(\mathcal{G}_{\alpha\alpha}P) + \partial_{\beta}(\mathcal{G}_{\alpha \beta}P) \\
\partial_\alpha(\mathcal{G}_{\beta \alpha}P) + \partial_{\beta}(\mathcal{G}_{\beta\beta}P )
\end{array}
\right],
\ee
with $g=\matrix{cc}{g_{\alpha 1}&g_{\alpha 2}\\ g_{\beta 1}&g_{\beta 2}}$.}
The reflecting boundary condition at $\beta=\pm 1$ reads:

\be
J_{\beta}(\alpha,\pm 1)\equiv 0\;\;\;\forall\alpha,
\ee 
where $J_{\beta}$ is the second (vertical) component of $\vec{J}$ in the rectangular coordinates.

The time-dependent solution {$P(\alpha,\beta,t)$} in the unwrapped phase does not approach a steady-state solution because the probability is constantly moving towards the right and diffusively spreading away from the center of mass.   
One can restore a steady-state solution by lumping all probability and all currents into one period (\textit{i.e.}~reverse the mapping from d to c in \bi{sketch_mapping}) by considering
\ba\quad\quad
p(\alpha,\beta,t)&=&\sum_{k=-\infty}^\infty P(\alpha+2k\pi,\beta,t) \xrightarrow{t\to\infty}\quad p_0(\alpha,\beta)\\
\quad\quad
\vec{\jmath}(\alpha,\beta,t)&=&\sum_{k=-\infty}^\infty \vec{J}(\alpha+2k\pi,\beta,t) \xrightarrow{t\to\infty}\quad \vec{\jmath}_0(\alpha,\beta)=(j_{\alpha,0}(\alpha,\beta),j_{\beta,0}(\alpha,\beta)).
\ea
The asymptotic (stationary) solution corresponds to the solution of \e{fpe_theta_r} with periodic boundary conditions over one period, e.g. $p_0(0,\beta)=p_0(2\pi,\beta)$. This follows from the smoothness of the function $ P(\alpha,\beta,t)$ as the solution of the Fokker-Planck equation under {our assumptions on} the coefficients.
The corresponding stationary probability current $\vec{j}_0=(j_{\alpha,0},j_{\beta,0})$ is related to the mean rotation rate as follows.  If we  integrate the $\alpha$-component of the current over the entire range of the radial variable, this yields the mean rotation rate or, equivalently, the inverse of the mean rotation time:
\be\label{eq:inverse-mean-rotation-time}
\overline{J}=\intl_{R_-}^{R_+}\!\! d\beta \; j_{\alpha,0}(\alpha,\beta)=\frac{1}{\tbar}.
\ee
Here $\alpha$ is an arbitrary but fixed value from the interval $[0,2\pi)$. Hence, one way to determine the mean rotation time is  to solve the stationary Fokker-Planck equation with periodic boundary conditions over the period and reflecting boundary conditions on the lower and upper boundaries in the radius and then to integrate the $\alpha$-component of the current over the radius. 

We {may} also write down the backward Kolmogorov equation for the probability density, which is a differential equation in terms of the initial position \rot{$\alpha_0,\beta_0$}, but we use the operator of this equation, the backward operator, ${\cal L}^\dagger$, in  a different context.  {Let $\tilde{\ell}$ be a line parameterized as $\tilde{\ell}=\{(\tilde{\alpha}(\tilde{\beta}),\tilde{\beta})\,:\, -1\le \tilde{\beta}\le -1\}$, with  $\tilde{\alpha}$ a $C^2$ function of $\tilde{\beta}$, and $\alpha<2\pi<\tilde{\alpha}(\beta)$, so that $(\alpha,\beta)$ is upstream of $\tilde{\ell}$.}
In the system with the unwrapped phase we can formulate the equation for the mean-first-passage time $\tilde{T}(\alpha,\beta)$ from initial position \rot{$(\alpha_0,\beta_0)$} {to such a line, with adjoint} reflecting boundary conditions \rot{at the interior and exterior boundaries of the cylindrical domain}
for $\beta=\pm 1$ and for $\alpha\to - \infty$ as follows:
\be\label{eq:tildeT}
{\cal L}^\dagger \tilde{T}(\alpha,\beta)=-1,
\ee 
where
\begin{align}\label{eq:Ldag_u}
{\cal L}^\dagger[u]=&\left[f_\alpha\frac\partial{\partial\alpha}+f_\beta\frac\partial{\partial\beta}+\frac12\left(\mathcal{G}_{\alpha\alpha}\frac{\partial^2}{\partial \alpha^2}+\mathcal{G}_{\alpha\beta}\frac{\partial^2}{\partial \alpha\beta}+\mathcal{G}_{\beta\alpha}\frac{\partial^2}{\partial \beta\alpha}+\mathcal{G}_{\beta\beta}\frac{\partial^2}{\partial \beta^2}\right)\right][u].
\end{align}
See \eqref{eq:bc-adjoint-reflect} for the adjoint reflecting (Neumann) boundary conditions.
Equation \eqref{eq:tildeT} holds true in the infinite strip and the values $(\alpha,\beta)$ can be as far off to the left as we wish.
In particular, 
because of the nonvanishing noise intensity, 
the solution $\tilde{T}(\alpha,\beta)$ is everywhere smooth and differentiable with respect to both variables \cite{mclean2000strongly}.

Next we give an intuitive derivation of the jump-periodic boundary condition the \rot{MRT} function should satisfy, which coincides with that established rigorously in \S\ref{sec:proof}. 
We cannot simply  lump everything back into one period as we did above with $P(\alpha,\beta,t)$ to obtain $p(\alpha,\beta,t)$, 
because $\tilde{T}$ is a time and not a probability density; a direct naive summation would certainly not make  sense. However, a kind of averaging summation of proper differences is meaningful as we will see in the following.
We first note that since the coefficients in the Langevin equations are all periodic in $2\pi$, this differential equation is identical for $\tilde{T}(\alpha,\beta)$ and $\tilde{T}(\alpha-2k\pi,\beta)$ ({i.e.}~for $k=1,2,3 \dots$, $\mathcal{L}^\dagger[\tilde{T}](\alpha,\beta)\equiv \mathcal{L}^\dagger[\tilde{T}](\alpha-2k\pi,\beta))$). 
Furthermore,  if the initial point is in our reference interval, \textit{i.e.}~$0\le \alpha\le 2\pi$, {we expect that} the time $\tilde{T}(\alpha-2k\pi,\beta)$ for $k\to\infty$  approaches the $k$-fold multiple of the mean rotation time $\tbar$ plus an order-one correction, 
{which would imply that} there exists a $T_{\rm max}$ such that 
\be
\left|\tilde{T}(\alpha-2k\pi,\beta)-k\tbar\right|<T_{\rm max}, \;\; \forall k \in \mathbb{Z}.
\ee
Then, the following sum of differences can be regarded as a kind of mean over the mean-first-passage times
\be
T_N(\alpha,\beta)=\frac{1}{N}\sum_{k=1}^N  (\tilde{T}(\alpha-2k\pi,\beta)-k\tbar).
\label{eq:sum}
\ee
The function $T_N(\alpha,\beta)$ obeys
\be
{\cal L}^\dagger T_N(\alpha,\beta)=-1.
\ee
because the constant terms ($k\tbar$) drop out (every term of the backward operator starts with a derivative) and each of the $N$ non-constant term contributes $-1/N$ to the right-hand side, \textit{i.e.}~in total the r.h.s. is -1 again.

We now define the function 
\be
T(\alpha,\beta)=\lim_{N\to\infty} T_N(\alpha,\beta)
\ee
as the limit for infinite summation {(assuming the limit exists)} and {argue} that this function  {should obey} a periodic-plus-jump condition. We first look at the boundary condition for $T_N$ for fixed values of $\alpha$ and $\beta$ with the reference interval
\ba
&&T_N(\alpha-2\pi,\beta)=\frac{1}{N}\sum_{k=1}^N  (\tilde{T}(\alpha-2(k+1)\pi,\beta)-k\tbar)=\frac{1}{N}\sum_{k=2}^{N+1}  (\tilde{T}(\alpha-2k\pi,\beta)-(k-1)\tbar)\nonumber\\
&&=\sum_{k=1}^{N} \frac{\tilde{T}(\alpha-2k\pi,\beta)-k\tbar}{N}+\tbar+\frac{\tilde{T}(\alpha-2(N+1)\pi,\beta)-(N+1)\tbar - (\tilde{T}(\alpha-2\pi,\beta)-\tbar)}{N}.\nonumber
 \ea
The last term vanishes in the limit $N\to\infty$ and we obtain
 \be
 T(\alpha,\beta)+\tbar=T(\alpha-2\pi,\beta).
 \label{eq:period_and_jump}
\ee
There is no constraint on the value of $\alpha$ except that it be smaller than $\tilde{\alpha}(\tilde{\beta})$ (parameterizing the absorbing boundary $\tilde{\ell}$) and thus we can express the function $T(\alpha,\beta)$ by a steady decrease in $\alpha$ and a purely periodic function in $\alpha$:
\be
T(\alpha,\beta)=U(\alpha,\beta)-\frac{\tbar}{2\pi} \alpha \;\; \mbox{with}\;\; U(\alpha,\beta)=U(\alpha+2k\pi,\beta),\; k\in \mathbb{Z}.
\ee
Two conclusions can be drawn.  First, we can solve 
\be\label{eq:backward_eqn_section_2}
{\cal L}^\dagger T(\alpha,\beta)=-1
\ee 
on the reference domain $[0,2\pi] \times [-1,1]$ with the periodic-plus-jump condition \e{period_and_jump} taken at $\alpha=0$  and adjoint-reflecting boundary conditions on the lines $\beta=\pm1$:
\ba
T(2\pi,\beta)+\tbar=&T(0,\beta), \;\; &\forall \beta\in[-1,1],\\
\left.{\cal L}_{\rm refl} T(\alpha,\beta)\right|_{\beta=\pm 1}=&0,  \;\; &\forall \alpha\in[0,2\pi].
\ea
{Here ${\cal L}_{\rm refl}u=0$ represents the adjoint reflecting boundary conditions, namely ${\cal L}_{\rm refl} = \mathcal{G}_{\beta\alpha}\partial_\alpha +\mathcal{G}_{\beta\beta}\partial_\beta $.}
As we establish in \S \ref{sec:proof}, under these conditions,
the function is only determined up to a constant and this constant is indeed the only trace left of the original absorbing boundary $\tilde{\ell}$.  Secondly, because \e{period_and_jump} holds true for all values of $\alpha$, the contour lines of $T(\alpha,\beta)=T_0$ fulfill the condition for an isochron $\rot{\ell_{\rm MFP}}$ by \spi, because from every point of the contour line $\ell_{\rm MFP}$
the mean time to the $2\pi$ shifted version of the contour line $\ell_{\rm MFP}'$ is equal to $\tbar$ by virtue of the periodicity condition \e{period_and_jump}, i.e.
\be
T(\alpha,\beta)\bigg|_{(\alpha,\beta)\in \ell_{\rm MFP}}= T(\alpha,\beta)\bigg|_{(\alpha,\beta)\in \ell_{\rm MFP}'}+\tbar.
\ee
\rot{
Finally, we note that the MRT phase $\Theta$ is given in terms of $T$ and an arbitrary constant $\Theta_0$ as
\be\label{eq:MRT-phase-in-alpha-beta}
\Theta(\alpha,\beta)=\Theta_0-T(\alpha,\beta)\frac{2\pi}{\overline{T}}.
\ee
With this definition, the MRT phase then obeys partial differential equation
\be
{\cal L}^\dagger \Theta=\frac{2\pi}{\overline{T}}
\ee 
with periodic-plus-jump condition $\Theta\to\Theta+2\pi$ along an arbitrary cut.
}
%

If the problem is given {originally} in polar coordinates (which will {be} the case for a few of our examples), \textit{i.e.}~if the starting point are the Langevin \e{langevin_polar}, we have solved the problem of finding an equation for the \rot{MRT} phase proposed by \spi. If the problem is stated originally in $(x,y)$, \e{langevin}, we reverse the mapping in \bi{sketch_mapping}, formulate the problem as a \rot{mean-return time} problem in the original domain with a periodic-plus-jump condition and find that the contour lines of the function $T(x,y)$ are the isochrons of the \rot{MRT} phase in the sense of \spi.  
\rot{Indeed, in accordance with equation \e{MRT-phase-in-alpha-beta}, the MRT phase in $(x,y)$ coordinates is given by 
\be\label{eq:MRT-phase-in-xy}
\Theta(x,y)=\Theta_0-T(x,y)\frac{2\pi}{\overline{T}}.
\ee}
Briefly, we first have to solve for the stationary probability density
\be
{\cal L} p_0(x,y)=0 \;\; \mbox{with} \;\; \left. {\cal R}_p p_0(x,y)\right|_{(x,y)\in R_\pm}=0 ,\; \int\int dx\; dy\; p_0(x,y)=1
\ee
from which we can extract the stationary current and the mean rotation period 
\be
\vec{\jmath}={\cal J} p_0(x,y)=(j_x(x,y),j_y(x,y)),\;\; \tbar=\frac{2\pi}{\int_{R_-(y=0)}^{R_+(y=0)} dx\, j_y(x,0)}.
\ee
{Here ${\cal R}_p p_0=\mathbf{n}_\pm^\intercal\nabla p_0=0$ gives the reflecting boundary conditions for the (forward) Fokker-Planck operator, where $\mathbf{n}_\pm$ is the local unit normal vector for the boundary at $(x,y)\in R_\pm$, respectively.}
We then solve the equation for the \rot{mean return time} in the original domain {with a jump across the line $R_-\le y\le R_+, x=0$},
\be\label{eq:this-is-the-equation-we-actually-solve-numerically}
{\cal L}^\dagger T(x,y)=-1\;\; \mbox{with} \;\;  \left. {\cal R}_{T} T(x,y)\right|_{(x,y)\in R_\pm}=0, \;\liml_{\varepsilon \to 0^+} \left(T(-\varepsilon,y)-T(\varepsilon,y)\right)\tbar.
\ee
\rot{The backward operator is written out in the supplement (\S\ref{supp:Stratonovich}) for both the Ito and Stra\-to\-novich interpretations of the original stochastic differential equations. Furthermore, in} \e{this-is-the-equation-we-actually-solve-numerically}
 ${\cal R}_{T}u= \sum_{j=1,2}n_j\sum_{k=1,2} \mathcal{G}_{jk}\partial_k u =0$ gives the adjoint reflecting or Neumann boundary conditions.
The isochrons of the \rot{MRT} phase are then given by the contour lines of $T(x,y)$ modulo $\tbar$ if the line crosses the boundary condition on $y=0$.  
\bi{het-finite-diff} \rot{(supplemental material)} illustrates the construction for a square domain \rot{used to solve \e{this-is-the-equation-we-actually-solve-numerically}.}

\rot{We round out our discussion of the MRT phase derivation with several comments:}\rm
\begin{remark} \rm
The jump condition can be imposed on any simple connection between the inner and outer boundaries. Furthermore, by the PDE and its boundary conditions, $T(x,y)$ is only determined up to an additive constant. 
\end{remark}
\begin{remark}\rm
\label{rem2} 
It may appear strange that we have to solve {\em two} PDEs (one for $p_0(x,y)$ and one for  $T(x,y)$) to solve the problem, especially because the first equation serves only to determine $\tbar$, the mean rotation time, which should be computable from the PDE for the \rot{MRT} in the first place. In fact, it can be shown by Green's-function techniques, that the boundary conditions of the PDE for $T(x,y)$ uniquely determine in a self-consistent manner $\tbar$. Thus, it would in principle suffice to solve only the PDE for $T(x,y)$. In all investigated examples we found, however, that the subsequent solution of the two equations  is numerically more practical and efficient.
\end{remark}
\begin{remark}\rm
The mapping to the annulus assumes a smooth boundary in the $(x,y)$ domain.  However, 
we do not see any  problem in principle to generalize the arguments used to a mapping that is only piecewise smooth. One of our examples, the heteroclinic oscillator on the square domain, is of such a type and does not seem to be problematic. 
\end{remark}
\begin{remark}\rm
Although, as elaborated in \S \ref{sec:proof} we require the noise to be nonsingular in order to rigorously establish existence and uniqueness of solutions, it is possible that this condition may not be strictly necessary.  For example,  in the numerical example with polarized noise below (\S \ref{ssec:y-polarized-noise}) the nonsingular noise condition is violated, nevertheless the numerical procedure based on the PDE appears to give the correct system of isochrons.
\end{remark}

\subsection{\rot{Relation to the asymptotic phase for deterministic systems}}

\rot{If a deterministic system of ordinary differential equations
\be\label{eq:limitcycle_ode}
\frac{d\mbx}{dt}=\mbf(\mbx),\quad\mbx\in\R^n
\ee
has a hyperbolically stable limit cycle solution $\gamma$ (a closed, isolated periodic orbit), with period $\tau$, we may define the phase $\Theta\in[0,2\pi)$ for points in $\Gamma=\{\gamma(t)\given 0\le t< \tau\}$ so that $\Theta=0$ for some reference point on $\Gamma$ and $d\Theta/dt=2\pi/\tau$.  We can extend this function to define the \emph{asymptotic phase} $\Theta$ for any point $\mbx_0$ in the basin of attraction (the stable manifold) of $\Gamma$, so that the trajectory starting at $\mbx_0$ converges to the periodic trajectory $\gamma(t+\tau \Theta(\mbx_0)/2\pi)$ as $t\to\infty$ \cite{ErmentroutTerman2010book,Guc75}.   Applying the chain rule, the asymptotic phase function must satisfy 
\be\label{eq:phase_limit_cycle}
\frac{d\Theta(\mbx)}{dt}=\mbf(\mbx)^\intercal \nabla\Theta(\mbx)=\frac{2\pi}{\tau}
\ee
for all $\mbx$ in the basin of attraction of $\Gamma$.  The boundary condition for this  linear first-order  and nonhomogeneous partial differential equation is set by continuity of $\Theta$ at the limit cycle $\Gamma$.

Now consider a family of stochastic differential equations of the form \e{langevin}, but with the noise scaled by a small parameter $\sqrt{\epsilon}$, and the corresponding family of solutions to the backward  \e{backward_eqn_section_2}, which we may write (for the planar case $\mbx\in\R^2$) as
\ba
\frac{d\mbx}{dt}&=&\mbf(\mbx)+\sqrt{\epsilon}\left(g_1(\mbx)\xi_1(t)+g_2(\mbx)\xi_2(t)  \right)\\
\label{eq:bkwd_epsilon_noise}
-1&=&\mathcal{L}_\epsilon^\dagger\left[T_\epsilon(\mbx)\right]=f(\mbx)^\intercal\nabla\left[T_\epsilon(\mbx)\right]+\frac\epsilon2\sum_{ij}\mathcal{G}_{ij}\partial^2_{ij}T_\epsilon(\mbx).
\ea
Suppose that, as $\epsilon\to 0$, $T_\epsilon(\mbx)$ converges uniformly on compact subsets of the domain to a $C^2$ function $T_0(\mbx)$.  Since for any $\epsilon$, $T_\epsilon$ is defined only up to an additive constant, we consider convergence in the sense that for arbitrary nonzero vectors $\mbv\in\R^2$, 
\be
\mbv^\intercal\left(\nabla T_0-\nabla T_\epsilon\right)\to 0,
\ee
for all $\mbx$ in the domain.  Fixing $\mbx$ and setting $\mbv= \mbf(\mbx)$, we see that  for each $\mbx$,
\be
\mbf^\intercal\nabla T_0-\mbf^\intercal\nabla T_\epsilon = 
\mbf^\intercal\nabla T_0-\left(-1-\frac\epsilon2\sum_{ij}\mathcal{G}_{ij}\partial^2_{ij}T_\epsilon\right)\to 0,\text{ as }\epsilon\to 0,
\ee
where we have used \e{bkwd_epsilon_noise}. Consequently if $T_\epsilon$ converges to a well-behaved function $T_0$ in this way, it must satisfy
\be\label{eq:time_function_limit_cycle}
\mathcal{L}_0^\dagger\left[T_0\right] = \mbf(\mbx)^\intercal\nabla\left[T_0(\mbx)\right]=-1. 
\ee
Comparing \e{phase_limit_cycle} and \e{time_function_limit_cycle}, evidently if \e{limitcycle_ode} has a stable limit cycle, then 
 the function $T_0$ must correspond with the deterministic asymptotic phase function $\Theta$ through the linear relation
 \be
 \Theta(\mbx)=\Theta_0-T_0(\mbx)\frac{2\pi}{\tau},
 \ee
 for arbitrary constant $\Theta_0$.
}


\section{Proof of existence and uniqueness of the isochron function $T$}
\label{sec:proof}

In this section we use Fredholm theory for strongly elliptic second order operators, and a maximum principle, to prove that the partial differential equation defining the \rot{mean--return-time} isochron function $T$ has a solution, and that the solution is unique up to an additive constant.  As in the previous section, we adopt coordinates $\alpha\in[0,2\pi)$ (the angular, periodic coordinate) and $\beta\in[-1,1]$ (the radial coordinate).  We will refer to the fundamental or local domain $\Omega=[0,2\pi)\times[0,1]$ and the extended domain $\Omega_\text{ext}=\R\times[-1,1]$.  
We assume the following:
\begin{enumerate}
\item[A1.] Transformed into $(\alpha,\beta)$ coordinates, the trajectories $(\alpha(t),\beta(t))$ of a strongly Markovian time-homogeneous process obey an Ito equation
\begin{equation}
\begin{split}
d\alpha &= f_1(\alpha,\beta)\,dt + g_{11}(\alpha,\beta)\,dW_1(t)+g_{12}(\alpha,\beta)\,dW_2(t)\\
d\beta &= f_2(\alpha,\beta)\,dt + g_{21}(\alpha,\beta)\,dW_1(t)+g_{22}(\alpha,\beta)\,dW_2(t)
\end{split}
\end{equation}
where $f_i, g_{ij}$ are $C^2$ on $\Omega_\text{ext}$.
\end{enumerate}
(Note: we will refer often below to the matrix $\mathcal{G}=gg^\intercal$, defined in terms of the $g_{ij}$.) 
\begin{enumerate}
\item[A2.] The functions $f_i$ and $g_{ij}$ are periodic in the first coordinate with period $2\pi$, \textit{i.e.}~$\forall \alpha\in\R$ and $i=1,2$, $f_i(\alpha+2\pi,\beta)=f_i(\alpha,\beta)$, and likewise for each $g_{ij}$.
\item[A3.] The second order differential operator $\mathcal{P}$ is strongly elliptic, where $\mathcal{P}$ is defined (following McLean)
\begin{equation}\label{eq:defineP}
\mathcal{P}u=-\sum_{j=1}^2\sum_{k=1}^2\partial_j(A_{jk}\partial_k u)+\sum_{j=1}^2 A_j\partial_j u \text{ on }\Omega_\text{ext}
\end{equation}
and $A_{jk}=-\frac{1}{2}\mathcal{G}_{jk}$ and $A_j=-\frac12\sum_{k=1}^2\partial_k\mathcal{G}_{jk}+f_j$.  
Along the lines $\beta=\pm1$ we impose Neumann boundary conditions 
\begin{align}\label{eq:bc-adjoint-reflect}
0=\sum_{j=1}^2\nu_j\sum_{k=1}^2 \mathcal{G}_{jk}\partial_k u
\end{align}
where $\nu=(0,\pm 1)$ is the outward unit normal at the respective boundary.
\end{enumerate}
We note that in the planar case ($n=2$) we consider here, strong ellipticity is guaranteed if the matrix 
\begin{equation}
\mathcal{G}=gg^\intercal=
\matrix{cc}{g_{11}^2+g_{12}^2  & g_{12}g_{22}+g_{21}g_{11}\\ 
g_{12}g_{22}+g_{21}g_{11} & g_{21}^2+g_{22}^2} 
\end{equation}
satisfies the nondegeneracy conditions
\begin{align}\label{eq:nondegeneracy}
\mathcal{G}_{11}>0,\quad \mathcal{G}_{22}>0,\text{ and }\det{\mathcal{G}}\not=0.
\end{align}

McLean's differential operator $\mathcal{P}$ corresponds to  Kolmogorov's backward operator $\mathcal{L}^\dagger$ occurring in the first-passage/return time problem, also known as the generator of the Markov process.  The adjoint $\mathcal{P}^*$ corresponds to Kolmogorov's forward operator (the Fokker-Planck operator) with reflecting boundary conditions at $\beta=\pm 1$, describing the evolution of probability densities forward in time.  We call \eqref{eq:bc-adjoint-reflect} \emph{adjoint reflecting} boundary conditions because $\mathcal{P}$ and $\mathcal{P}^*$ are adjoint operators on the appropriate function spaces.  Specifically, $\mathcal{P}$ acts on $L^\infty(\Omega)\cap C^2(\Omega)$ (bounded, twice differentiable functions) while $\mathcal{P}^*$ acts on $L^1(\Omega)\cap C^2(\Omega)$ (twice differentiable functions integrable on the local domain).

Note that, although we consider two independent noise sources, $dW_1$ and $dW_2$, we could also have $k>2$ independent noise sources (in this case $g$ would be an $n\times k$ matrix) without a fundamental change in the results, provided $\mathcal{G}$ is nonsingular.  

We further assume that:
\begin{enumerate}
\item[A4.] The process viewed on $\Omega$ (taking $\alpha$ mod $2\pi$) admits a density $\rho(\alpha,\beta,t)$ evolving according to 
\begin{equation}\label{eq:adjoint-pde}
\frac{\partial\rho}{\partial t}=\mathcal{L}\rho=\mathcal{P}^*\rho=-\sum_{j=1}^2\sum_{k=1}^2\partial_j(A_{kj}^*\partial_k \rho)-\sum_{j=1}^2\partial_j(A_j^*\rho)
\end{equation}
where $A_{kj}^*=A_{jk}$ and $A_j^*=A_j$, and we impose reflecting (Neumann) boundary conditions 
\begin{equation}\label{eq:adjoint-reflecting-bc}
0=\sum_{k=1}^2A_{kj}^*\partial_k\rho+A_j^*\rho
\end{equation}
at $\beta=\pm 1$ (for all $\alpha$), and periodic boundaries in $\alpha$, i.e.
\begin{equation}\label{eq:adjoint-periodic-bc}
\forall \beta\in[-1,1],\quad \rho(0,\beta)=\rho(2\pi,\beta).\end{equation}
We assume that the system has a unique stationary distribution $\rhoss\ge 0$, with $1=\int_\Omega\,d\alpha\,d\beta\rhoss(\alpha,\beta)$, \textit{i.e.}~satisfying the homogeneous equation
\begin{equation}\label{eq:adjointhomog}
\mathcal{P}^*\rhoss=0,
\end{equation}
together with the boundary conditions \eqref{eq:adjoint-reflecting-bc}-\eqref{eq:adjoint-periodic-bc}.
\end{enumerate}
We note the stationary flux vector $J_\text{ss}(\alpha,\beta)$ corresponds componentwise to 
\begin{equation}\label{eq:ssflux}
J_{\text{ss},j}(\alpha,\beta)=f_j\rhoss-\frac12\sum_{k=1}^2\partial_k\left(\mathcal{G}_{jk}\rhoss\right).
\end{equation}
Lastly, we assume the mean drift is nonzero and oriented to the right:
\begin{enumerate}
\item[A5.] If $\gamma:[-1,1]\to[0,2\pi]$ is any $C^1$ function whose graph $C_\gamma=\{(\alpha=\gamma(\beta),\beta)\; : \; -1\le \beta \le 1\}$ connects the inner and outer domain, separating $\Omega_\text{ext}$ into  left and  right connected components, with unit normal $\mathbf{n}(\beta)$ oriented into the right connected component, then the mean rightward flux through $C_\gamma$ is positive, \textit{i.e.}
\begin{equation}\label{eq:meanJ}
0<\overline{J}:=\int_{-1}^1 d\beta\,\mathbf{n}^\intercal(\beta) J_\text{ss}(\gamma(\beta),\beta).
\end{equation}
\end{enumerate}
The reciprocal of $\overline{J}$ is proportional to the mean period of the oscillator, 
\begin{equation}\label{eq:tbar}
\tbar=\left(\overline{J}\right)^{-1}.
\end{equation}

Our goal is to establish the existence of a function $T$ satisfying the inhomogeneous PDE with adjoint reflecting boundary conditions at $\beta=\pm 1$ and jump-periodic boundary conditions, namely
\begin{align}\label{eq:PT1}\nonumber
\mathcal{P}T= -1,&\quad\text{on }\Omega\\
\sum_{k=1}^2 \mathcal{G}_{2k}\partial_k T(\alpha,\pm 1)=0,&\quad\forall\alpha\in\R\\
T(\alpha,\beta)-T(\alpha+2\pi,\beta)=\tbar,&\quad\forall (\alpha,\beta)\in\Omega_\text{ext}.\nonumber
\end{align}

\begin{theorem}[Existence and Uniqueness of Solutions]\label{thm:main}
If Assumptions A1-A5 hold, then the partial differential equation \eqref{eq:PT1} with reflecting adjoint boundary conditions at $\beta=\pm 1$ and jump-periodic boundary conditions in the $\alpha$ coordinate  has a solution $T(\alpha,\beta)$ on $\Omega_\text{ext}$ and, by restriction, on $\Omega$.  Moreover, the solution is unique up to an additive constant.  
\end{theorem}

The proof relies on Theorem 4.10 of \cite{mclean2000strongly} (stated as Theorem \ref{thm:McLean4.10} below), a version of the Fredholm alternative for strongly elliptic PDEs with general boundary conditions.  This theorem employs the following notation.  For a Lipschitz domain $\Omega$, and $s\in\R$,
$H^s(\Omega)$  denotes the Sobolev space of order $s$ (based on the $L_2$ norm).   $H_D^s(\Omega)$ denotes the subspace of functions in $H^s(\Omega)$ that equal zero when restricted to the portion of the boundary of $\Omega$ where a Dirichlet condition is enforced.  Since our boundaries do not include a Dirichlet component, for our problem $H_D^s(\Omega)\equiv H^s(\Omega)$.
In addition, $\widetilde{H}^{s}(\Omega)$ denotes the closure of $C_\text{comp}^\infty(\Omega)$ in $H^s(\R^n)$ (for the planar systems we consider, $n=2$) where $C_\text{comp}^\infty(\Omega)$ is the space of $C^\infty$ functions with compact support in $\Omega$. 
The space $H^{-1}(\Omega)$ is the set of distributions (or functions) $h(\alpha,\beta)$ that can be represented as $h(\alpha,\beta)=\partial_\alpha h_1(\alpha,\beta)+\partial_\beta h_2(\alpha,\beta)$ for some functions $h_1,h_2\in L_2(\Omega)$.  The inhomogeneities $h(\alpha,\beta)$ we consider will be $C^2$ functions of the coordinates, hence integrable over the compact domain $\Omega$, and members of $H^{-1}(\Omega)$. For further details, see \cite{mclean2000strongly}.

Theorem 4.10 will require that the operator $\mathcal{P}$ be \emph{coercive} on $H^s_D(\Omega)$.  Coercivity is a growth rate condition that holds for our system by virtue of the following  
\begin{theorem}[McLean 2000, Thm 4.7]\label{thm:McLean4.7}
Assume that $\mathcal{P}$ has scalar coefficients, and that $\mathcal{P}$ is strongly elliptic on $\Omega$.  If the leading coefficients satisfy
$$A_{kj}=A_{jk}\quad\text{ on }\Omega,\quad\text{ for all }j\text{ and }k,$$
then $\mathcal{P}$ is coercive on $H^1(\Omega)$.
\end{theorem}
The symmetry condition $A_{kj}=A_{jk}$ holds for our system because $\mathcal{G}=gg^\intercal$ is a symmetric matrix.  Strong ellipticity follows from our assumptions A1-A5, and will be established in the proof of our main theorem (\ref{thm:main}) below.

Finally, we will refer to the homogeneous problem 
\begin{align}\nonumber\label{eq:homog}
\mathcal{P}u&=0,\quad\text{on }\Omega\\
\sum_{k=1}^2 \mathcal{G}_{2k}\partial_k u(\alpha,\pm 1)&=0,\quad\forall\alpha\in[0,2\pi]\\
u(0,\beta)-u(2\pi,\beta)&=0,\quad\forall\beta\in[-1,1]\nonumber 
\end{align}
associated with the general inhomogeneous problem (with periodic boundary conditions)
\begin{align}\nonumber\label{eq:inhomog}
\mathcal{P}u&=h,\quad\text{on }\Omega\\
\sum_{k=1}^2 \mathcal{G}_{2k}\partial_k u(\alpha,\pm 1)&=g_N,\quad\forall\alpha\in[0,2\pi]\\
u(0,\beta)-u(2\pi,\beta)&=0,\quad\forall\beta\in[-1,1]\nonumber 
\end{align}

McLean's theorem 4.10 is more general than our problem requires.  We specialize to the case of scalar-valued functions ($m=1$, below), and our boundary contains no Dirichlet component ($\Gamma_D=\emptyset$).  Nevertheless for ease of comparison we state McLean's version of the theorem in full:
\begin{theorem}[McLean 2000, Thm 4.10]\label{thm:McLean4.10}
Assume that $\Omega$ is a bounded Lipschitz domain, and that $\mathcal{P}$ is coercive on $H_D^1(\Omega)^m$.  Let $h\in\widetilde{H}^{-1}(\Omega)^m, g_D\in H^{1/2}(\Gamma_D)^m$ and $g_N\in H^{-1/2}(\Gamma_N)^m,$ and let $W$ denote the set of solutions in $H^1(\Omega)^m$ to the homogeneous problem \eqref{eq:homog}.
There are two mutually exclusive possibilities:
\begin{enumerate}
\item[(i)] The homogeneous problem has only the trivial solution, i.e., $W=\{0\}$.  In this case,  the homogeneous adjoint problem \eqref{eq:adjointhomog}, with boundary conditions \eqref{eq:adjoint-reflecting-bc}-\eqref{eq:adjoint-periodic-bc}, 
also has only the trivial solution in $H^1(\Omega)^m$, and for the inhomogeneous problem \eqref{eq:inhomog} 
we get a unique solution $u\in H^1(\Omega)^m$.  Moreover,
\begin{equation}  
||u||_{H^1(\Omega)^m}\le C||h||_{\widetilde{H}^{-1}(\Omega)^m}+C||g_D||_{H^{1/2}(\Gamma_D)^m}+C||g_N||_{H^{-1/2}(\Gamma_N)^m}.  
\end{equation}
\item[(ii)] The homogeneous problem has exactly $p$ linearly independent solutions, \textit{i.e.} $\dim W=p$, for some finite $p\ge 1$.  In this case, the homogeneous adjoint problem \eqref{eq:adjointhomog}, with boundary conditions \eqref{eq:adjoint-reflecting-bc}-\eqref{eq:adjoint-periodic-bc},  
also has exactly $p$ linearly independent solutions, say $v_1,\ldots,v_p\in H^1(\Omega)^m$, and the inhomogeneous problem \eqref{eq:inhomog} 
is solvable in $H^1(\Omega)$ if and only if
\begin{equation}  \label{eq:thm4p13_solvability}
(v_j,h)_\Omega+(\gamma v_j,g_N)_{\Gamma_N}=\left( \widetilde{B}_\nu v_j,g_D \right)_{\Gamma_D}\quad\text{ for }1\le j \le p. 
\end{equation}
\end{enumerate}
\end{theorem}
In \eqref{eq:thm4p13_solvability} $\gamma$ refers to the ``trace operator" that maps a function $v$ to its value on the boundary.

Before proving Theorem \ref{thm:main} we state and prove a lemma concerning the solutions of the homogeneous problem, exploiting a maximum principle.
\begin{lemma}\label{lem:const}
If Assumptions A1-A5 hold, then every $C^2$ solution $u$ of the homogeneous  problem \eqref{eq:homog} is a constant function.
\end{lemma}
\begin{proof}[Proof of Lemma \ref{lem:const}]
Suppose a $C^2$ function $u(\alpha,\beta)$ satisfying \eqref{eq:homog} on the domain 
$\Omega=[0,2\pi]\times[-1,1]$ attains a strict maximum $u_+$ at a point $(\alpha_0,\beta_0)$ in the interior of $\Omega$. At this point $\partial_\alpha u=\partial_\beta u=0$.   Thus $\sum_iA_i\partial_i u\equiv0$ at $(\alpha_0,\beta_0)$.  Therefore (referring to \eqref{eq:defineP} of A3)
\begin{align}
0&=\sum_{i,j}A_{ij}\partial_i\partial_j u
= \text{tr}(\mathcal{A}(\alpha_0,\beta_0)H(\alpha_0,\beta_0)),\quad\text{ where}\\ \nonumber
\mathcal{A}&=\left(A_{ij}\right)\text{ and }\\ \nonumber
H&=\left(\partial_i\partial_j u\right)\text{ is the Hessian matrix}.
\end{align}
By our strong ellipticity assumption (A3), $\det \mathcal{A}>0$.  If $\alpha_0,\beta_0$ is a strict local maximum then at this point $\det(H)>0$.  Therefore $\det{\mathcal{A}H}>0$ as well.  Both $\mathcal{A}$ and $H$ are symmetric.  By the trace identity, 
$\text{tr}(\mathcal{A}H)=\text{tr}\left(\mathcal{A}^{1/2}H\mathcal{A}^{1/2}\right)$, so $\text{tr}(\mathcal{A}H)>0$, contradicting $\text{tr}(\mathcal{A}H)=0$.  Therefore $u$ cannot achieve a strict local maximum (or, by a parallel argument, a strict local minimum) anywhere in the interior of $\Omega$, so $u$ must be constant.
\end{proof}

Rather than proving the existence of the function $T(\alpha,\beta)$ directly, we prove the existence of a  function $U(\alpha,\beta)$ satisfying a related inhomogeneous PDE with periodic boundary conditions, and obtain $T$ from the equation
\begin{equation}\label{eq:TfromU}
T(\alpha,\beta)=U(\alpha,\beta)-\frac{\tbar}{2\pi}\alpha.
\end{equation} 

\begin{proof}[Proof of Theorem \ref{thm:main}]
Assuming A1-A5 hold, consider the inhomogeneous problem 
\begin{align}\nonumber\label{eq:inhomogU}
\mathcal{P}U=-1-f_1(\alpha,\beta)\frac{\tbar}{2\pi},&\quad\text{ on }\Omega\\ 
\sum_{k=1}^2 \mathcal{G}_{2k}\partial_k U(\alpha,\pm 1)=\pm\mathcal{G}_{21}(\alpha,\pm1)\frac{\tbar}{2\pi},&\quad\forall\alpha\in[0,2\pi]\\
U(0,\beta)-U(2\pi,\beta)=0,&\quad\forall\beta\in[-1,1]\nonumber 
\end{align}
which is equivalent to the system \eqref{eq:inhomog} (the object of Theorem \ref{thm:McLean4.10}) with inhomogeneities $h=-1-f_1\tbar/2\pi$ and $g_N=\pm\mathcal{G}_{21}\tbar/2\pi$ in the interior and on the reflecting boundary of $\Omega$, respectively.  The $\pm$ in the inhomogeneous boundary condition reflects the direction of the outward unit normal at the upper boundary, $\nu=(0,1)$ at $\beta=1$, \textit{versus} the lower boundary, $\nu=(0,-1)$ at $\beta=-1$.

To see that Theorem \ref{thm:McLean4.10} applies, note that 
\begin{enumerate}
\item $\Omega=[0,2\pi]\times[-1,1]$ is a bounded Lipschitz domain, by construction; 
\item the condition $h\in\widetilde{H}^{-1}(\Omega)^m$ holds since $f_1$ is assumed to be a $C^2$ function on a compact domain $\Omega$;
\item the condition $g_D\in H^{1/2}(\Gamma_D)^m$ is satisfied trivially since $\Gamma_D$ (the subset of the boundary on which a Dirichlet condition is imposed) is  the empty set; 
\item the condition $g_N\in H^{-1/2}(\Gamma_N)^m$ holds because $\mathcal{G}_{21}$, is assumed to be a $C^2$ function on the compact domain $\Omega$, \textit{a fortiori} on its boundary $0\le\alpha\le2\pi$, $\beta=\pm 1$; and
\item $\mathcal{P}$ is coercive, by Assumption A3 and Theorem \ref{thm:McLean4.7}.
\end{enumerate}
Thus we may apply Theorem \ref{thm:McLean4.10} to the system \eqref{eq:inhomogU}.
By Lemma \ref{lem:const}, all solutions of the homogeneous problem \eqref{eq:homog} are constant.  On the other hand, any function $u=\text{const}$ satisfies \eqref{eq:homog}. Hence $\dim W = p=1$, and case (ii) of Theorem \ref{thm:McLean4.10} applies.  Therefore, the inhomogeneous problem \eqref{eq:inhomogU} has solutions if and only if 
\begin{equation}\label{eq:solvability}
0=(v_1,h)_\Omega+(\gamma v_j,g_N)_{\Gamma_N},
\end{equation}
where $v_1$ is any solution of the adjoint homogeneous problem \eqref{eq:adjointhomog} with boundary conditions \eqref{eq:adjoint-reflecting-bc}-\eqref{eq:adjoint-periodic-bc}.  By Assumption A4, all solutions of the adjoint homogeneous problem are multiples of the steady state density $\rhoss$, so the solvability condition \eqref{eq:solvability} is equivalent to
\begin{equation}\label{eq:solvability2}
\begin{split}
0=&\int_\Omega d\alpha\,d\beta\,\rhoss(\alpha,\beta)\left(-1+\frac{\tbar}{2\pi}f_1(\alpha,\beta)\right) \\
&-\frac12\int_{\alpha=0}^{2\pi}d\alpha\,\left[ \rhoss(\alpha, 1)\mathcal{G}_{21}(\alpha,1)\frac{\tbar}{2\pi}-\rhoss(\alpha,-1)\mathcal{G}_{21}(\alpha,-1)\frac{\tbar}{2\pi}\right].
\end{split}
\end{equation}
{But the steady-state probability flux \eqref{eq:ssflux} satisfies
\begin{align}
\frac{2\pi}{\tbar}=&\int_\Omega d\alpha\,d\beta\,J_{\text{ss},1}\\ \nonumber
=&\int_\Omega d\alpha\,d\beta\left[f_1\rhoss-\frac12\left(
\frac{\partial}{\partial\alpha}\left(\mathcal{G}_{11}\rhoss\right)+
\frac{\partial}{\partial\beta}\left(\mathcal{G}_{12}\rhoss\right)
\right)
 \right]\\ \label{eq:tbar-middle-term}
 =&\E[f_1]-\frac12
\int_{\beta=-1}^1 d\beta\,\left(\mathcal{G}_{11}(2\pi,\beta)\rhoss(2\pi,\beta)-\mathcal{G}_{11}(0,\beta)\rhoss(0,\beta)\right)\\
\label{eq:tbar-last-term}
&-\frac12\int_{\alpha=0}^{2\pi} d\alpha\,\left(\mathcal{G}_{12}(\alpha,1)\rhoss(\alpha,1)-\mathcal{G}_{12}(\alpha,-1)\rhoss(\alpha,-1)\right).
\end{align}
The first and last terms in \eqref{eq:tbar-middle-term}-\eqref{eq:tbar-last-term} match corresponding terms in \eqref{eq:solvability2}, and the middle term in \eqref{eq:tbar-middle-term} vanishes because both $\mathcal{G}$ and $\rhoss$ are periodic in $\alpha$. Therefore we may rewrite \eqref{eq:solvability2} in terms of the steady-state flux:}
\begin{equation}
\tbar=\frac{2\pi}{\int_\Omega d\alpha\,d\beta\, J_{\text{ss},1}}.
\end{equation}
{Thus the mean period $\tbar$ emerges as a parameter required to have a specific value in order that the solvability condition should be satisfied.  But this condition is already given in Assumption A5.  (Compare Remark \ref{rem2} on page \pageref{rem2}.)}

We conclude that the solutions of the inhomogeneous problem for $U$ \eqref{eq:inhomogU} have the form $U(\alpha,\beta)+c$ for arbitrary $c\in\R$.  Having established the existence of the function $U$, it remains to show that the function $T$ defined by \eqref{eq:TfromU} solves the inhomogeneous problem \eqref{eq:PT1} with jump-periodic boundary conditions.  From \eqref{eq:defineP} it is clear that $\mathcal{P}(c)=0$ and $\mathcal{P}(\alpha)=f_1(\alpha,\beta)$.  By inspection of \eqref{eq:inhomogU}, it is evident that $\mathcal{P}(T)=-1$, as required.  The adjoint reflecting boundary conditions are satisfied, since clearly
\begin{equation}
\left[\mathcal{G}_{21}\frac{\partial}{\partial \alpha} + \mathcal{G}_{22}\frac{\partial}{\partial \beta} \right](\alpha)=\mathcal{G}_{21}.
\end{equation}
The jump-periodic boundary condition on the local domain, $T(0,\beta)=T(2\pi,\beta)+\tbar$, holds by construction.  On the extended domain $T(\alpha,\beta)=T(\alpha+2\pi,\beta)+\tbar$ for all $(\alpha,\beta)\in\Omega_\text{ext}$.

Finally, $T$ inherits uniqueness from $U$, up to an additive constant.
\end{proof}  

Remark: The properties of strong ellipticity and coercivity are preserved under $C^1$ diffeomorphisms (\cite{mclean2000strongly}, p.~156).  If we begin with a star-shaped domain and make the transformation described in \S \ref{ssec:strip} we thus establish the existence of a unique $T$-function (up to an additive constant) the level curves of which satisfy the \rot{MRT} property postulated by \spi.  

For an extension of theorem \ref{thm:main} to oscillators in $n>2$ dimensions, see \S \ref{sec:ndim}.

\section{Detailed examples}
\label{sec:examples}
We illustrate the solution of the backward equation and the resulting \rot{MRT} isochrons for  several example systems \rot{(one additional example may be found in supplemental section \S \ref{ssec:counterrot}).}
\rot{A detailed discussion of the numerical method may be found in supplemental section \S \ref{sec:supp_numerical}.}

\subsection{An analytically solvable test case - the \rot{isotropic} noisy Stuart-Landau oscillator}\label{ssec:clock}
We consider a solvable case with a trivial solution for the phase, here simply given by the angle of the common polar coordinates (the simple geometric phase). The system of stochastic differential equations is in this case given in polar coordinates by 
\begin{align}
\dot{\theta} &= \omega  + \sqrt{2D} \xi_1 (t),\quad
 \label{eq:isochronal}
\dot{r} = -\gamma r (r^2 - 1) + \sqrt{2D} \xi_2 (t).
\end{align}
Here $\xi_{1,2}(t)$ are independent Gaussian white noise sources with $\lr{\xi_{i}(t)\xi_{j}(t')}=\delta_{i,j}\delta(t-t')$. 
With $\gamma>0$ the noiseless system has a stable fixed point in the radius variable at $r=1$; with noise the radius fluctuates around this value and there are also fluctuations in the phase velocity (see \bi{clock-jump-results}a for a trajectory plotted in cartesian coordinates) but there is no coupling between phase and amplitude. We solve the \rot{MRT} equation numerically for a bounded annular region cutting out an inner circle and restricting the motion to a disk with a maximum radius.  
\begin{figure}[h!]
\centerline{
\includegraphics[width=\linewidth]{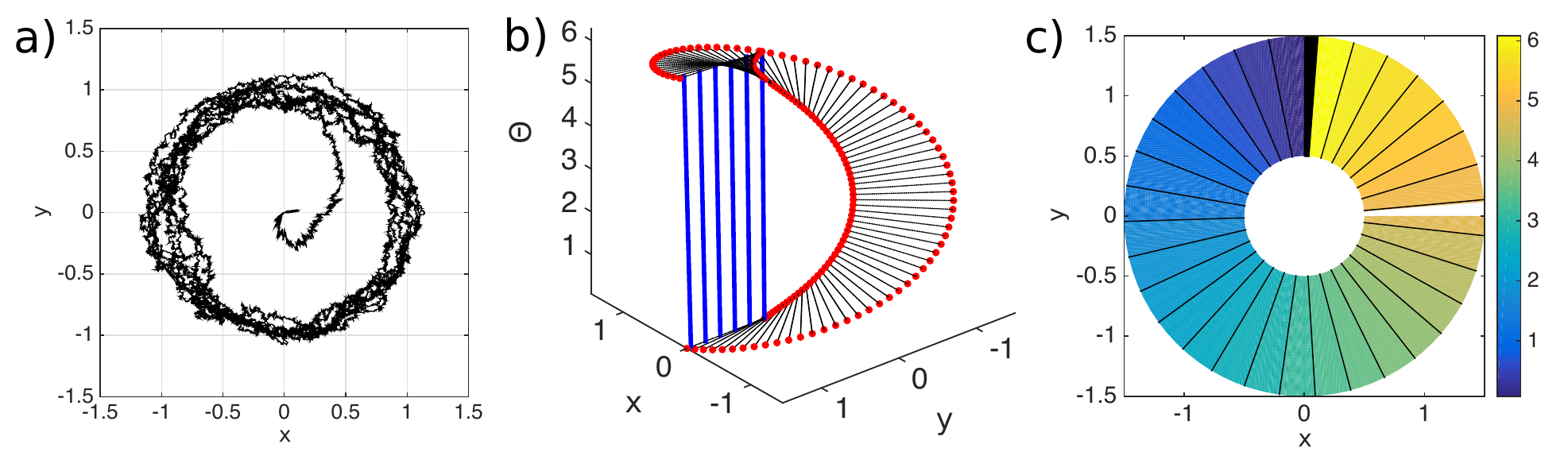}
}
 \caption{Isochronal Stuart-Landau oscillator clock model, \e{isochronal}, with $\omega=1$, $\gamma = 1$, and $D=0.01$. 
 Sample trajectory (with timestep $\Delta t=0.001$) with  initial condition at $(0,0)$ (a), \rot{MRT phase $\Theta(x,y)$} obtained by  a finite difference scheme (b) and its contour lines (c), which are the \rot{MRT} isochrons. As expected, the isochrons  are spokes-of-the-wheel and each ``spoke" is constant up to the order of $10^{-12}$. \rot{For the parameters, used $\overline{T}=2\pi$.}}

 \label{fig:clock-jump-results}
\end{figure}

If the drift term of the angular variable (in polar coordinates) does not depend on the radius, it is easy to show that  the \rot{MRT} function $T(\theta,r)$ is linear in $\theta$, \textit{i.e.}~
\be
T(\theta,r)=T_0-\frac{\tbar}{2\pi}\theta=T_0-\frac{\theta}{\omega}\rot{,\quad\text{therefore }\Theta(\theta,r)=\theta+\Theta_0}
\ee
(where $T_0$ \rot{and $\Theta_0$ are arbitrary constants}) is a solution of the PDE that satisfies the periodic-plus-jump condition:
the radial part of the backward operator yields immediately zero, the linear increase in $\theta$ results in the inhomogeneity on the right hand side.
The contour lines of this function\rot{, and hence also the MRT phase function $\Theta(x,y)$,} are simply the spokes of a wheel and this is also returned by our numerical routine, which solves the PDE numerically  (see \bi{clock-jump-results}, middle and right). These would be also the isochrons for the deterministic oscillator. \rot{Our numerical solution, obtained for  finite inner and outer boundaries, reproduces this trivial analytical solution, which is a first indication of the robustness of the solution method.}

\subsection{Stuart-Landau oscillator with antirotating  phase}
\label{ssec:antirot}
The next example is \rot{due to Newby and Schwemmer \cite{NewSch14} and displays an interesting} amplitude-phase coupling:
\begin{align}
\dot{\theta} &= \omega  \rot{+} \omega\gamma c(1-r)^2+ \sqrt{2D} \xi_1 (t),\quad  
\label{eq:antirot}
\dot{r} = -\gamma r (r^2 - 1) + \sqrt{2D} \xi_2 (t)
\end{align}
With parameters \rot{$\gamma>0$ and $c<0$}, the radius-dependent drift in the $\theta$ dynamics will cause an anti-rotation compared to the movement on the limit cycle at $r=1$ whenever the radius sufficiently deviates  from the limit cycle (both for smaller and larger values). The exact condition for this change in the deterministic force field (cf. also \bi{antirotate-results}a) is $(1-r)^2>\rot{(-\gamma c)^{-1}}$.
\begin{figure}[h!]
\centerline{
  \includegraphics[width=\linewidth]{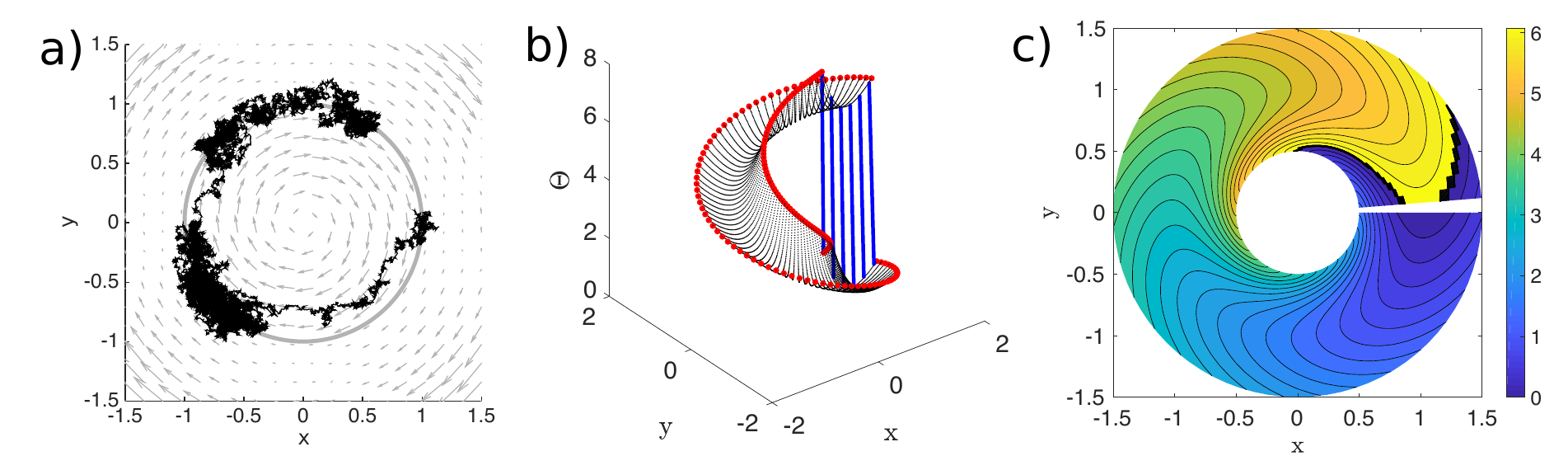} 
}
  \caption{Stuart-Landau oscillator with antirotating phase, \e{antirot}, with $\omega=1$, $\gamma = 15$, $\rot{c=-15}$, and $D=0.198$.
Sample trajectory  ($\Delta t=10^{-4}$) for an initial condition at $(1,0)$, limit cycle (grey bold line) and vector field shown by arrows (a), \rot{MRT phase $\Theta(x,y)$}  from the finite-difference-scheme solution (b) and its contour lines, \textit{i.e.}~the isochrons of the \rot{MRT} phase (c). \rot{For the parameters used, $\overline{T}=11.89$.} 
}
  \label{fig:antirotate-results}
\end{figure}
The shape of the isochrons reflects the reversed deterministic velocity away from the limit cycle. \rot{Both positions} 
inside and outside the limit cycle \rot{have mean angular velocity larger than, and in the same direction as, the mean angular rotation, while points near the limit cycle have mean angular velocity that is smaller and in the reverse direction.  Consequently,} the isochrons, resulting as contour lines from the PDE solution, attain a hook-like shape similar to what is observed for the deterministic system.
As  the noise becomes stronger, we expect that the isochrons of the \rot{MRT} phase become more straight because the trajectory diffuses faster in the radial direction, which makes all the initial points for the race around the circle more equal. Put differently, with more noise in the system, less head start is required for extreme radial positions to complete one round in the same mean time.

\subsection{Stuart-Landau oscillator with $y$-polarized noise}
\label{ssec:y-polarized-noise}
The previous examples shared a basic rotational symmetry: the dynamics in polar coordinates was independent of the geometric angle $\theta$. As a consequence, all the isochrons can be mapped onto each other by a simple rotation. 

\begin{figure}[h!]
\centerline{
\includegraphics[width=\linewidth]{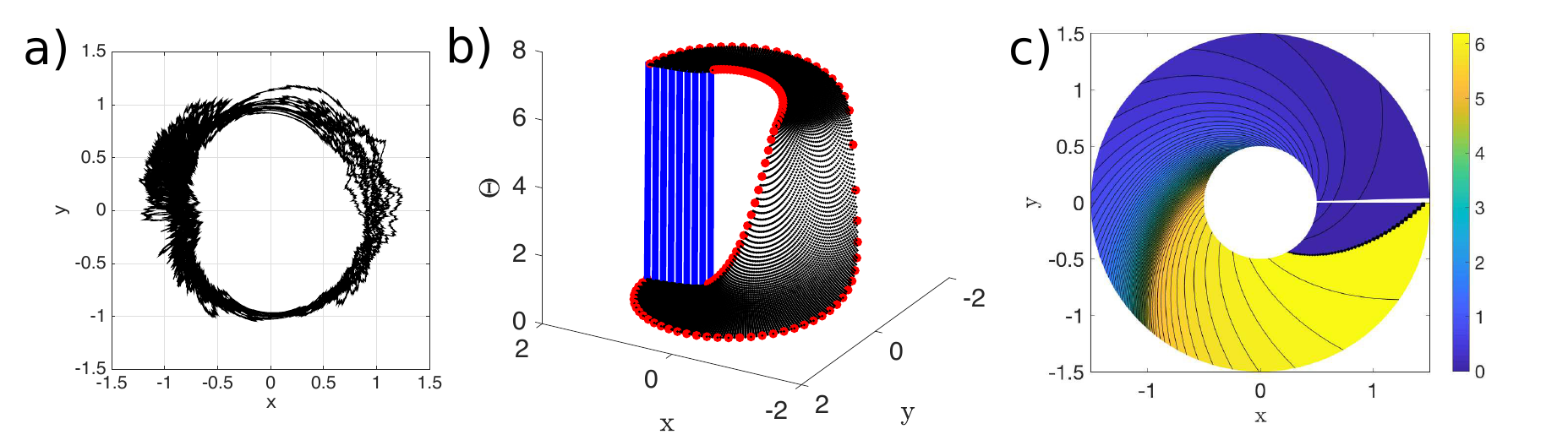} 
}
  \caption[Noisy Stuart-Landau oscillator with $y$-polarized noise \rot{MRT} and average isophases]{Noisy Stuart-Landau oscillator with $y$-polarized noise given by \e{y-noise} with $\omega=1.99$, $\kappa=1$, and $\sigma=0.2$. Example trajectory (a)   ($\Delta t=0.01$, trajectories generated in polar $(r,\theta)$ coordinates and then converted to Cartesian $(x,y)$ for analysis), \rot{MRT phase $\Theta(x,y)$}  (b) and its contour lines, \textit{i.e.}~isochrons of the \rot{MRT} phase (c). \rot{For the parameters used, $\overline{T}=20.93$.} 
  } 
  \label{fig:y-polar-SL-results}
\end{figure}

We consider now a case that lacks this symmetry, the Stuart-Landau oscillator with $y$-polarized noise\footnote{\spi~study this system using the equivalent Stratonovich formulation, compare equation (5) of \cite{SchPik13}, $\dot{\theta}=\omega+r\cos\theta-\kappa r^2+\sqrt{2D}\sin(\theta)\circ\xi(t)$, 
$\dot{r}=r(1-r^2)+\sqrt{2D} r\cos(\theta)\circ\xi(t)$.}:
\begin{equation}\label{eq:y-noise}
\begin{split}
\dot{\theta}&=f_1(\theta,r)+\sqrt{2D}\sin(\theta)\xi(t)\\
\dot{r}&=f_2(\theta,r)+\sqrt{2D} r \cos(\theta)\xi(t)\\
f_1(\theta,r)&=\omega+r\cos\theta-\kappa r^2+D\cos\theta\sin\theta\\
f_2(\theta,r)&=r(1-r^2)+rD\left(\cos^2\theta-\sin^2\theta \right)
\end{split}
\end{equation}
Note  that the $\theta$ drift is now also modified such that the angular velocity is sped up for $\theta \approx \pi$ but slowed down around $\theta \approx 0$.  In addition \eqref{eq:y-noise} is an example of an excitable system: in the absence of noise the dynamics have a stable fixed point rather than a limit cycle, and a circulating current appears only in the presence of noise. 

Note also that the same noise sample $\xi(t)$ drives both the $\theta$ and $r$ coordinate.  Thus in contrast to \eqref{eq:isochronal} and \eqref{eq:antirot} the noise is singular, \textit{i.e.}~the nondegeneracy condition $\det\mathcal{G}\not=0$ (eq.~\eqref{eq:nondegeneracy}) is violated. The  algorithm  nevertheless produces the \rot{MRT} phase  (Fig.~\ref{fig:y-polar-SL-results}).

The \rot{MRT} phase resulting from the PDE solution displays a rather heterogeneous structure. First of all, there is a clustering of isochrons on the left, where the angular velocity is systematically increased as explained above. Secondly, there is \rot{little} noise in the $\theta$ variable \rot{when} the trajectory is close to the $x$ axis; \rot{at the same time, there is little} noise in the radial dynamics close to the y-axis, because of the sine and cosine prefactors of the noise terms. 

Consequently, the deviation of the isochrons from the spokes-of-a-wheel shape changes significantly: on the right hand side, isochrons cover a larger range of the angular variable than on the left hand side, resulting from an interplay between the angle dependence of phase velocity and noise level.  

\begin{figure}[h!]
  \includegraphics[width=\linewidth]{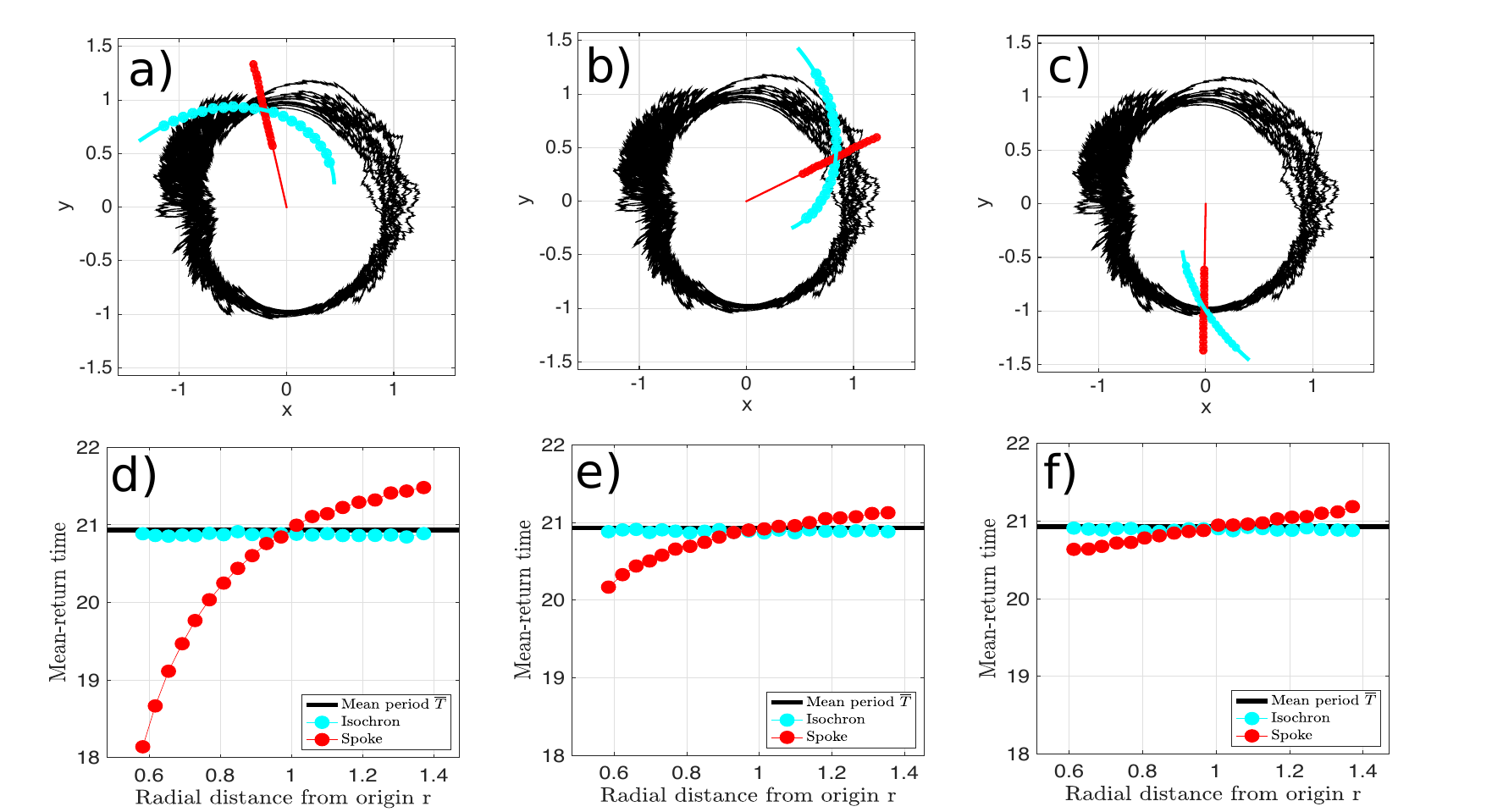} 
\caption{Testing the \rot{MRT} property for  the noisy Stuart-Landau oscillator with $y$-polarized noise \eqref{eq:y-noise}. Sample trajectory (a-c, black trace) and parameters are as in \bi{y-polar-SL-results}.  \rot{Cyan} curves (a-c) show isochronal sections obtained from the PDE solution; red curves are ``spoke" sections passing through the isochronal section. For each of 20 points per curve, equally spaced in the radial coordinate, we generate approximately 600,000 independent trajectories, using the Euler-Maruyama method (trajectories generated in polar $(r,\theta)$ coordinates with $\Delta t=0.001$) and computing the \rot{MRT} from a starting point to the same section after completing one rotation (d-f). For initial conditions on the PDE-derived isochrons, \rot{mean return} times (d-f, \rot{cyan dots}) are within 0.5\% of the mean period (d-f, \rot{black} horizontal line).  For initial conditions on a spoke, \rot{MRT}s deviate significantly from the mean period (d-f, \rot{red dots}).
Thus  the isochronal sections obtained as level curves of $T(x,y)$ solved with the periodic-plus-jump condition, is a \rot{MRT} isochron in the sense of \spi, whereas the spoke shows a strong increase of the mean return time with growing radius and is thus not such an isochron. }
\label{fig:MFRTs-compare}
\end{figure}

We take this most involved example also as a test case to check whether the algorithmic definition of the \rot{MRT} phase is matched by our PDE result (\bi{MFRTs-compare}). We pick out {three} contour lines of the PDE solution {(\bi{MFRTs-compare}a-c)}, distribute a number of initial points along this line (\rot{cyan} dots), and start an ensemble of trajectories on each of the  points (\bi{MFRTs-compare}b). We simulate one rotation until the \rot{respective} trajectory hits the back of the isochron; the corresponding mean return times are plotted in  \bi{MFRTs-compare}d-f and {reveal} a very good agreement for all initial points - the picked contour lines
\rot{satisfy} \spi's algorithmic definition of the \rot{MRT property}. For comparison we also show the same measurement {in each case} for another line, the spoke of a wheel passing through the same part of the noisy limit cycle (red {lines} in \bi{MFRTs-compare}a-c). Points on {these lines} have different mean return times to the same line, hence, the spoke is not an isochron of the \rot{MRT} phase {in any of the  cases considered}.

\subsection{Noisy heteroclinic oscillator} 
\label{ssec:heteroclinic}
Our last example is given by the stochastic differential equations 
\begin{align}
\begin{split}
\label{eq:hetero}
\dot{x} & = \cos(x) \sin(y) + \alpha \sin(2x) + \sqrt{2D} \xi_1 (t), \\
\dot{y} & = -\sin(x) \cos(y) + \alpha \sin(2y)+ \sqrt{2D} \xi_2 (t). 
\end{split}
\end{align}
This model has been used in Ref.~\cite{ThoLin14} to study another notion of phase, the asymptotic phase of a stochastic oscillator, and is an example of a noisy heteroclinic network (\cite{ArmbrusterStoneKirk:2003:Chaos,Bakhtin2010ProbThyRelatFields,StoArm99,StoneHolmes1990SIAMJApplMath}, see also \cite{HirschSmaleDevaney2004}). The model's deterministic dynamics ($D=0$) is characterized by four saddle points in the corners $(\pm \pi/2,\pm \pi/2)$
and an unstable focus at $(0,0)$. Without noise the system approaches the heteroclinic cycle connecting the saddles and the system does not sustain a finite period oscillation (the heteroclinic cycle has an infinite period). With a weak noise, the system displays pronounced oscillations of a finite phase coherence \cite{ThoLin14} and each of the components has pronounced peaks in its power spectrum \cite{BalTho17}.
In the plane these oscillations become manifest by a clockwise rotation around the origin (\bi{het_jump_results}a).

Here we solve \e{this-is-the-equation-we-actually-solve-numerically} on the domain $(-\pi/2,\pi/2)\times (-\pi/2,\pi/2)$ and cut out a small square in the middle. On these outer and inner boundaries reflecting boundary conditions are imposed while along the blue cut line in \bi{het-finite-diff}a,  we apply the jump condition. The resulting  mean return time (as a function of the starting position $(x,y)$) is shown in  \bi{het_jump_results}, middle and the isoclines, which are the isochrons of the \rot{MRT} phase, are displayed in \bi{het_jump_results}, right.

\begin{figure}[h!]
  \includegraphics[width=\linewidth]{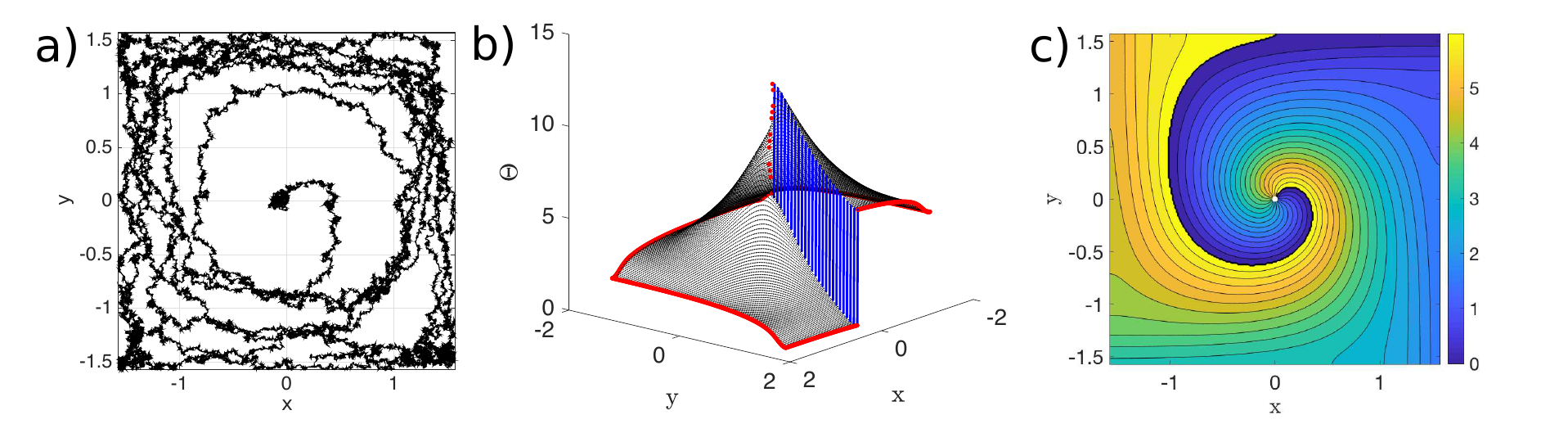} 
\caption[Noisy heteroclinic oscillator \rot{MRT} and average isophases]{Noisy heteroclinic oscillator, \e{hetero}, with $\alpha=0.1$, and $D=0.01125$.  Sample trajectory  ($\Delta t=0.001$)  with initial condition at $(0,0)$ that  moves clockwise (a). In the absence of noise, the system has a stable heteroclinic orbit with infinite period. With noise ($D>0$), the small perturbations eventually knock the trajectory out of the corners to form a stochastic oscillator with finite mean period. \rot{MRT phase $\Theta(x,y)$} finite-difference-scheme solution (b) and its contour lines, \textit{i.e.}~the isochrons of the \rot{MRT} phase (c). The computational grid is $251\times 251$ and the punctured square has side length of 0.05.
\rot{For the parameters used, $\overline{T}=16.23$.}
}
\label{fig:het_jump_results}
\end{figure}

The \rot{MRT} isochrons wind inward around the origin in a counterclockwise direction. In terms of the mean return time this geometry has a simple interpretation: we have to start at an advanced position if we start at the slow track on the outside (close to the square's sides), where the speed is reduced, compared to a starting position on the inside (close to the central square), where the speed is larger.   

\begin{figure}[htbp] 
   \centering
   \includegraphics[width=0.8\linewidth]{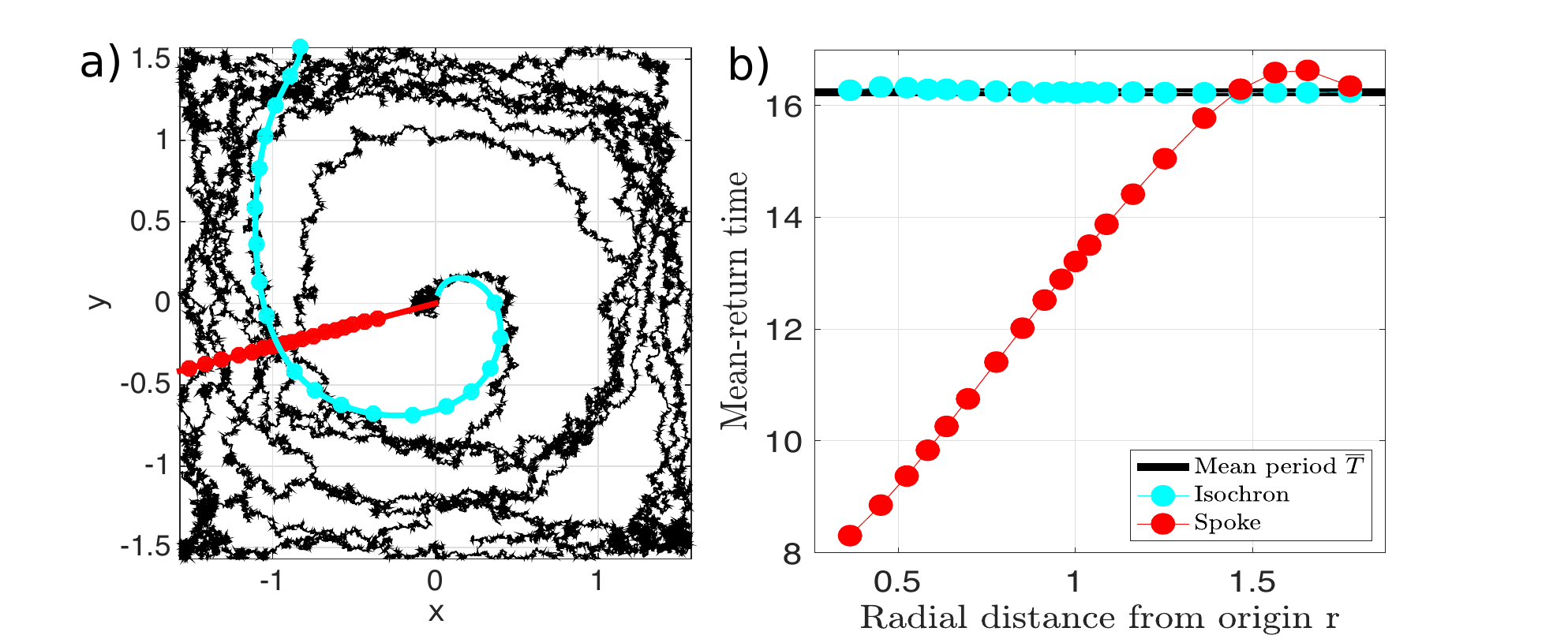}  
   \caption{(a) Heteroclinic trajectory in black for $D=0.01125$ and $\alpha=0.1$ with initial condition at origin. The \rot{cyan} curve is an isophase pulled from the  contour plot in \bi{het_jump_results}c. The \rot{cyan dots}  indicate $20$ equally spaced initial conditions along this curve. The red curve is a ``spoke" running through the tenth isophase initial condition. The red \rot{dots} indicate $20$ equally spaced initial conditions. (b) The \rot{black} line is the mean period calculated from the stationary flux (used in the finite difference scheme). \rot{Cyan dots} are the \rot{MRT}s back to the isophase after completing one full oscillation of the isophase initial conditions, from an ensemble of approximately 520,000 trajectories generated using the Euler-Maruyama method  with $\Delta t=0.001$.  The $x$-axis is the radial distance from the origin. \rot{Red dots}  are the \rot{MRT}s back to the spoke after completing one full \rot{rotation} from the spoke initial conditions.    }
   \label{fig:heteroclinic-compare}
\end{figure}

As in the case of the Stuart-Landau oscillator with $y$-polarized noise, we can test that the level curves of the solution of \e{this-is-the-equation-we-actually-solve-numerically} have the \rot{MRT} property.  \bi{heteroclinic-compare} contrasts the \rot{mean return} times for ensembles of trajectories with initial conditions located either along a level curve of \rot{$\Theta$, the MRT phase}, \textit{vs.}~along a simple radial section.   Using an ensemble of independent trajectories for each initial condition, the \rot{mean return} times for the spoke vary by a factor of two as the starting distance from the origin ranges from 0.5 to 1.5.  In contrast, the \rot{MRTs} for the isophase fall within $0.7\%$ of the mean period regardless of starting radius, illustrating \spi's \rot{MRT} criterion for an isochron.

\section{Extension to $n$-dimensional oscillators}
\label{sec:ndim}
For continuity of presentation and ease of illustration, we gave our main theorem \ref{thm:main} for an oscillator in two dimensions.  But the proof based on McLean's Theorem 4.10 carries over to oscillators in arbitrary finite dimension with little modification, so it is worth stating the general case here.  As before, we assume that our noisy oscillator has been mapped through a coordinate transformation to a system with one periodic coordinate $\alpha\in S^1\equiv [0,2\pi]$ and the remaining $n-1$ coordinates lying within \rot{a simply connected compact subset of} 
$\R^{n-1}$, \rot{with piecewise smooth boundary; call this set $B_1^{n-1}\subset\R^{n-1}$.} 
Thus the 
region is assumed to have the form of an $(n-1)$-dimensional
\rot{cylinder} 
embedded in $\R^n$.  
We define the local domain $\Omega=S^1\times B_1^{n-1}$ and the extended domain $\Omega_\text{ext}=\R\times B_1^{n-1}$. The driving noise vector $\mathbf{dW}$ will have $K\ge n$ components $dW_1,\ldots,dW_K$.  Recapitulating assumptions A1-A5 in this context, we assume the following:
\begin{enumerate}
\item[A1'.] Transformed into $(\alpha,\vec{\beta})$ coordinates, the trajectories $(\alpha(t),\vec{\beta}(t))$ of a strongly Markovian time-homogeneous process obey an Ito equation
\begin{equation}
\begin{split}
d\alpha &= f_1(\alpha,\vec{\beta})\,dt + \sum_{k=1}^K g_{1k}(\alpha,\vec{\beta})\,dW_k(t)\\
d\beta_i &= f_i(\alpha,\vec{\beta})\,dt + \sum_{k=1}^Kg_{ik}(\alpha,\vec{\beta})\,dW_k(t),\text{ for }2\le i \le n,
\end{split}
\end{equation}
where $f_i, g_{ik}$ are $C^2$ on $\Omega_\text{ext}$ for $1\le i\le n$ and $1\le k \le K$.\\
As before, we will refer to the $n\times n$ matrix $\mathcal{G}=gg^\intercal$.
\item[A2'.] The functions $f_i$ and $g_{ik}$ are periodic in the first coordinate with period $2\pi$, \textit{i.e.}~$\forall \alpha\in\R$ and $i=1,2$, $f_i(\alpha+2\pi,\vec\beta)=f_i(\alpha,\vec\beta)$, and likewise for each $g_{ik}$.
\item[A3'.] The second order differential operator $\mathcal{P}$ is strongly elliptic, where $\mathcal{P}$ is defined 
\begin{equation}\label{eq:defineP_nD}
\mathcal{P}u=-\sum_{i=1}^n\sum_{j=1}^n\partial_i(A_{ij}\partial_j u)+\sum_{i=1}^nA_i\partial_i u \text{ on }\Omega_\text{ext}
\end{equation}
and $A_{ij}=-\frac{1}{2}\mathcal{G}_{ij}$ and $A_i=-\frac12\sum_{j=1}^2\partial_j\mathcal{G}_{ij}+f_i$.  
At the boundary $|\vec\beta|=1$ we impose Neumann boundary conditions 
\begin{align}\label{eq:bc-adjoint-reflect-nD}
0=\sum_{i=1}^2\nu_i\sum_{j=1}^2 \mathcal{G}_{ij}\partial_j u
\end{align}
where $\nu$ is the outward unit normal at the boundary.
\item[A4'.] The process viewed on $\Omega$ (taking $\alpha$ mod $2\pi$) admits a density $\rho(\alpha,\vec\beta,t)$ evolving according to 
\begin{equation}\label{eq:adjoint-pde-nD}
\frac{\partial\rho}{\partial t}=\mathcal{L}\rho=\mathcal{P}^*\rho=-\sum_{i=1}^n\sum_{j=1}^n\partial_i(A_{ji}^*\partial_j \rho)-\sum_{i=1}^2\partial_i(A_i^*\rho)
\end{equation}
where $A_{ji}^*=A_{ij}$ and $A_i^*=A_i$, and we impose reflecting (Neumann) boundary conditions 
\begin{equation}\label{eq:adjoint-reflecting-bc-nD}
0=\sum_{j=1}^2A_{ji}^*\partial_j\rho+A_i^*\rho
\end{equation}
at $|\vec\beta|=1$ (for all $\alpha$), and periodic boundaries in $\alpha$, i.e.
\begin{equation}\label{eq:adjoint-periodic-bc-nD}
\forall \vec\beta\in B_1^{n-1},\quad \rho(0,\vec\beta)=\rho(2\pi,\vec\beta).\end{equation}
We assume that the system has a unique stationary distribution $\rhoss\ge 0$, with $1=\int_\Omega\,d\alpha\,d\vec\beta\rhoss(\alpha,\vec\beta)$, \textit{i.e.}~satisfying the homogeneous equation
\begin{equation}\label{eq:adjointhomog-nD}
\mathcal{P}^*\rhoss=0,
\end{equation}
together with the boundary conditions \eqref{eq:adjoint-reflecting-bc}-\eqref{eq:adjoint-periodic-bc}.
\end{enumerate}
As before, the stationary flux vector $J_\text{ss}(\alpha,\vec\beta)$ corresponds componentwise to 
\begin{equation}
J_{\text{ss},i}(\alpha,\vec\beta)=-\frac12\sum_{j=1}^n\partial_j\left(\mathcal{G}_{ij}\rhoss\right)+f_i\rhoss.
\end{equation}
\begin{enumerate}
\item[A5'.] We assume the mean drift is nonzero and oriented in the direction of increasing $\alpha$.  That is, if $\gamma:B_1^{n-1}\to[0,2\pi]$ is any $C^1$ function whose graph $C_\gamma=\{(\alpha=\gamma(\vec\beta),\vec\beta)\; : \; \vec\beta\in B_1^{n-1}\}$ cuts transversely through the extended domain, separating $\Omega_\text{ext}$ into  left and  right connected components, with unit normal $\mathbf{n}(\vec\beta)$ oriented into the downstream connected component, then the mean rightward flux through $C_\gamma$ is positive, \textit{i.e.}
\begin{equation}\label{eq:meanJ-nD}
0<\overline{J}:=\int_{-1}^1 d\vec\beta\,\mathbf{n}^\intercal(\vec\beta) J_\text{ss}(\gamma(\vec\beta),\vec\beta).
\end{equation}
\end{enumerate}
These assumptions suffice to establish the existence and uniqueness (up to an additive constant) of a \rot{MRT} function $T(\alpha,\vec\beta)$ satisfying $\mathcal{L}^\dagger[T]=-1$ subject to adjoint-reflecting boundary conditions $\sum_{i=1}^n\nu_i(\alpha,\vec\beta)\mathcal{G}_{ij}\partial_jT=0$ \rot{at the boundary $\partial \Omega$,} and jump-periodic boundary conditions $\forall \vec\beta\in B_1^{n-1}$, $T(\alpha,\vec\beta)=T(\alpha+2\pi,\vec\beta)+\tbar$, where $\tbar$ is the mean period of the oscillator, $\tbar=\left(\overline{J}\right)^{-1}$.  As in the planar case, the proof would \rot{involve} the auxiliary function $U(\alpha,\vec\beta)$ satisfying \begin{equation}
\label{eq:inhomogU-nD}
\begin{split}
\mathcal{P}U=-1-f_1(\alpha,\vec\beta)\frac{\tbar}{2\pi},&\quad\text{ on }\Omega\\ 
\sum_{i=2}^n\nu_i\sum_{j=1}^n \mathcal{G}_{ij}\partial_j U(\alpha,\pm 1)=\mathcal{G}_{i1}(\alpha,\vec\beta)\frac{\tbar}{2\pi},&\quad\forall\alpha\in[0,2\pi]\text{ and }\rot{\vec\beta\in\partial B_1^{n-1}}\\
U(0,\vec\beta)-U(2\pi,\vec\beta)=0,&\quad\forall\vec\beta\in B_1^{n-1},
\end{split}
\end{equation}
with $\nu$ the outward unit normal vector \rot{at the boundary of $\Omega$}. 

\section{Summary and conclusions}
\label{sec:conclusions}
In this paper we have found an analytic way to define the mean--return-time (\rot{MRT}) phase, originally proposed by Schwabedal and Pikovsky in terms of an algorithm, for the important class of smooth two-dimensional stochastic oscillators that are driven by white Gaussian noise.
We showed that the defining isochrons are given as the contour lines of the solution of the conventional PDE for the mean-first-passage-time function but with an uncommon periodic-plus-jump  condition. We illustrated this construction in a number of stochastic oscillator models and verified the algorithmic \rot{MRT}-property of our PDE solution for the most involved {examples} (Stuart-Landau oscillator with y-polarized noise {and heteroclinic oscillator}). Some open questions remain.

It would be of interest to study the effect of increasing levels of noise on the shape of the isochrons. Preliminary results indicate that in many cases the shape of the isochrons change towards spokes of a wheel with increasing noise level, \textit{i.e.}~the isochrons become less curved. {A similar phenomenon was observed in a version of the asymptotic phase for stochastic oscillators that we introduced previously, based on a spectral decomposition of the generator \cite{ThoLin14}.} However, it is unclear what are the exact conditions for the stochastic oscillator (\textit{i.e.}~for the functions appearing in the Langevin \e{langevin}) {leading to isochrons of greater or lesser  curvature}.
Also, in cases in which no deterministic phase exists as for the heteroclinic oscillator, one should also consider the opposite limit and let the noise level shrink. 
  
Finally, one should compare systematically to the asymptotic phase proposed in \cite{ThoLin14}. For the trivial case of the isotropic Stuart-Landau oscillator that lacks an angle-amplitude coupling, both definitions yield the same stochastic phase which is the same as in the deterministic case \cite{PhysRevLett.115.069402}. Generally, in cases 
with an existing deterministic phase, both definitions of a stochastic phase yield this deterministic phase in the limit of vanishing noise and hence can be regarded as possible generalizations of the deterministic phase to the stochastic case.  For a finite noise level, however, there seem to be small differences between the asymptotic phase and the \rot{MRT} phase in most systems. The exact nature of these differences and their role in reduced descriptions of stochastic oscillators remain exciting topics of future research.

 \section{Acknowledgments}
  \rot{The authors thank J.~Dobrosotskaya, W.~McLean, and D.~Gurarie for helpful discussions. PT acknowledges support from NSF grant DMS-1413770, from Oberlin College,
  and from the Mathematical Biosciences Institute and the NSF under grant DMS-1440386.}
 Large-scale Monte Carlo simulations made use of the High Performance Computing Resource in the Core Facility for Advanced Research Computing at Case Western Reserve University.

\bibliographystyle{siamplain}
\bibliography{ALL_19_03_11,references,math}

\newpage

\begin{center}
SUPPLEMENTARY MATERIALS
\end{center}

\section{Numerical Method for Obtaining the MRT Phase}
\label{sec:supp_numerical}

In this section we describe in detail how we incorporate the jump condition into a finite difference numerical scheme to obtain the mean--return-time function $T(x,y)$ for the heteroclinic oscillator example of \S\ref{ssec:heteroclinic}, and we briefly investigate the robustness of the isochrons to changes in the inner domain boundary size and specification of the mean period $\tbar$.  

\subsection{Finite Difference Method}
\label{ssec:supp_finite_diff}

We solve equation \eqref{eq:this-is-the-equation-we-actually-solve-numerically}: 
$$
{\cal L}^\dagger T(x,y)=-1\;\; \mbox{with} \;\;  \left. {\cal R}_{T} T(x,y)\right|_{(x,y)\in R_\pm}=0, \;\liml_{\varepsilon \to 0^+} \left(T(-\varepsilon,y)-T(\varepsilon,y)\right)=\tbar,
$$
with adjoint reflecting boundary conditions, i.e.
$${\cal R}_{T}u= \sum_{j=1,2}n_j\sum_{k=1,2} \mathcal{G}_{jk}\partial_k u =0$$
along the inner and outer boundary of a square domain 
$$\mathcal{D}=\left([-\pi/2,\pi/2]\times[-\pi/2,\pi/2]\right)\backslash\left([-\epsilon/2,\epsilon/2]\times[-\epsilon/2,\epsilon/2]\right),$$
with a jump of $\tbar$ implemented along a vertical cut at $x=0$ running from $y=\epsilon/2$ to $y=\pi/2$, see  \bi{het-finite-diff}A.
\begin{figure}[htbp] 
   \centering
   \includegraphics[width=\linewidth]{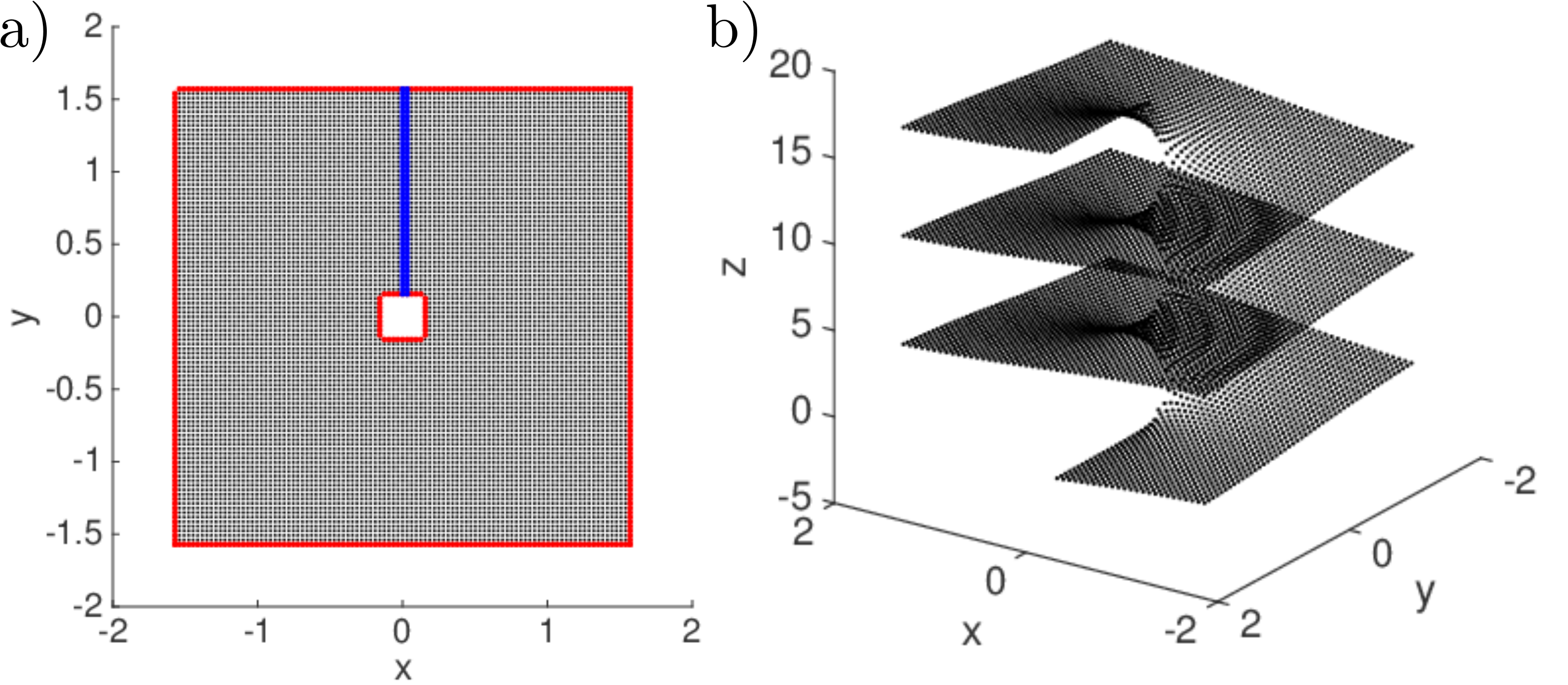} 
   \caption{Numerical scheme (heteroclinic oscillator example). 
   (a) Finite difference grid for the heteroclinic oscillator.  Interior points (black grid) use standard second-order finite difference operators to approximate first and second derivatives. Adjoint reflecting boundary conditions are imposed at the inner and outer boundaries (red points). The size of the inner boundary is exaggerated for clarity.  The jump condition is imposed along a vertical cut (blue points).  The computational domain  for \bi{het_jump_results} used a $251\times 251$ grid with spacing $\Delta x=\Delta y=\pi/250$ and inner square side length $\epsilon=2\pi/125\approx0.05$.     
   (b) A ``time crystal", a la Winfree \cite{Winfree1980}, providing a didactic illustration of the unwrapping of the domain to form a multi-valued surface,  analogous to the unwrapped domain introduced in \S\ref{ssec:unwrapped} of the main text.
   Coordinates $x$ and $y$ correspond to $(x,y)$ coordinates in panel A; coordinate $z$ corresponds to the (multivalued) geometric phase $z(x,y)=\tan^{-1}(y/x)$.  The jump condition eliminates the necessity of considering multiple copies of the domain.
   }
   \label{fig:het-finite-diff}
\end{figure}

We employed a standard finite difference scheme to solve the MFPT-PDE with the mean period jump condition. The domain was uniformly discretized and the points within the punctured center removed. All derivatives were approximated using the 2nd order, centered finite differences at interior points of the domain (black grid points in \bi{het-finite-diff}A), excluding the cut line (blue points in \bi{het-finite-diff}A):
\begin{align}
\label{eq:supp_diff1}
\left(\partial_x T \right)_{i,j} &\to \frac{T_{i+1, j} - T_{i-1, j}}{2h} \\
\label{eq:supp_diff2}
\left(\partial^2_{xx} T \right)_{i,j} &\to \frac{T_{i+1, j} - 2T_{i,j} + T_{i-1, j}}{h^2} \\
\label{eq:supp_diff3}
\left(\partial^2_{xy} T \right)_{i,j} &\to \frac{T_{i+1, j+1} - T_{i+1,j-1} - T_{i-1, j+1} + T_{i-1, j-1}}{4h^2}
\end{align}
and similarly for the $y$ derivatives. 

At inner and outer boundary points (red points in \bi{het-finite-diff}A) the adjoint reflecting boundary conditions were implemented by introducing ``ghost points" one grid space beyond the edge of the domain.  The normal vector $(n_1,n_2)^\intercal$ was taken to be $(\pm 1,0)^\intercal$ or $(0,\pm 1)^\intercal$ respectively, as dictated by the geometry.  

We implemented the mean period jump condition along a line segment connecting the outer and inner boundaries (blue dots in \bi{het-finite-diff}A). For the heteroclinic oscillator implementation, the cut was made along the vertical line segment from $(x=0, y=\epsilon/2)$ to $(x=0, y=\pi/2)$. Denote the set of finite difference points on this line segment by writing $(i_0,j) \in \ell$. Then, for $(i_0,j)$ we implement the jump in $T$ between the point $(i_0,j)$ and the immediately adjacent point to the left, which we denote $(i_0+1,j)$. 
That is, we replace $T_{i_0+1,j}$ with $\left(\tbar+T_{i_0+1,j}\right)$.  For these points we thus replace \eqref{eq:supp_diff1}-\eqref{eq:supp_diff2} with
\begin{align}
\label{eq:supp_diff1_jump}
\frac{T_{i_0+1, j} - T_{i_0-1, j}}{2h} &\to \frac{\left(\tbar+T_{i_0+1, j}\right) - T_{i_0-1, j}}{2h} \\
\label{eq:supp_diff2_jump}
\frac{T_{i_0+1, j} - 2T_{i_0,j} + T_{i_0-1, j}}{h^2} &= \frac{\left(\tbar +T_{i_0+1, j} \right) - 2T_{i_0,j} + T_{i_0-1, j}}{h^2}. 
\end{align}

A parallel procedure of subtracting $\overline{T}$ from $T_{i_0, j}$ is used for derivatives centered at location $(i_0+1, j)$, whenever location $(i_0,j) \in \ell$.

The resulting system of equations is linear-affine, with all constants involving $\overline{T}$ written on the right-hand side.

As a technical point, the system is rank-deficient, leading to nonuniqueness of solutions.  This point is readily addressed by specifying the value of $T$ at a particular location $(i_*,j_*)$. We prescribed the value of the northwest corner when solving for $T$ for the heteroclinic oscillator:
\begin{align}
T\left( - \frac{\pi}{2}, \frac{\pi}{2} \right) = T_{0,0} = 100
\end{align}
Fixing $T$ at this point removes one unknown from the linear system, and the resulting reduced-dimensional system is uniquely solvable with standard linear algebraic methods.

The examples in \S\ref{ssec:clock}, \S\ref{ssec:antirot}, \S\ref{ssec:y-polarized-noise}  and \S\ref{ssec:counterrot} were implemented using standard polar coordinates rather than  Cartesian coordinates.   
All examples in the main text used a $251\times 251$ discretization of the computational domain.  
Code to reproduce the examples is available at 
\texttt{https://github.com/pjthomas9/isophase}.

In each of these examples, the annular computational domain contains the overwhelming bulk of the stationary probability distribution.  As observed by \spi, the definition of the MRT phase presumes that the oscillator trajectories avoid the region near the ``center" of the domain.  The inner boundary should be made small enough that the excised region at the center contains very little net probability.  Similarly, if the underlying model has an unbounded domain, then the outer boundary should be large enough that the region excluded from the computational domain contains very little net probability.

\subsection{Robustness of the Method}

\subsubsection{Tolerance to the value of $\tbar$}

Although we did not systematically investigate parametric robustness of the method in detail, our heuristic observations suggest the method is robust against some variation in the $\tbar$ parameter.  That is, we have observed that if the parameter representing $\tbar$ varies by up to 5\% from the true value of $\tbar$, the method nevertheless produces a reasonably smooth solution that is quantitatively close to the solution produced with the correct value of $\tbar$.

\subsubsection{Sensitivity to inner domain size}

In order to test the sensitivity of the method for the heteroclinic oscillator system to changes in the size of the inner domain boundary, we ran the numerical algorithm using $\epsilon\in\{0.05,0.1,0.2\}$ on a $201\times 201$ finite difference map, with the other parameters the same as given in \S\ref{ssec:heteroclinic}.
\bi{boundary_sensitivity} compares the MRT function $T(x,y)$ obtained when the interior boundary region $[-\epsilon/2,\epsilon/2]\times[-\epsilon/2,\epsilon/2]$ is imposed with different values of $\epsilon$.  Both the absolute difference in the value of $T$ (panels A-B) and the shape of the isochrons (panels C-D) show modest sensitivity to $\epsilon$ restricted to a small (roughly $O(\epsilon)$ in size) neighborhood of the inner domain boundary.  Outside a small neighborhood the values of $T$ and the shape of the isochrons are insensitive to the precise value of $\epsilon$ dictating the inner boundary size.
\begin{figure}[htbp] 
   \centering
   \includegraphics[width=\linewidth]{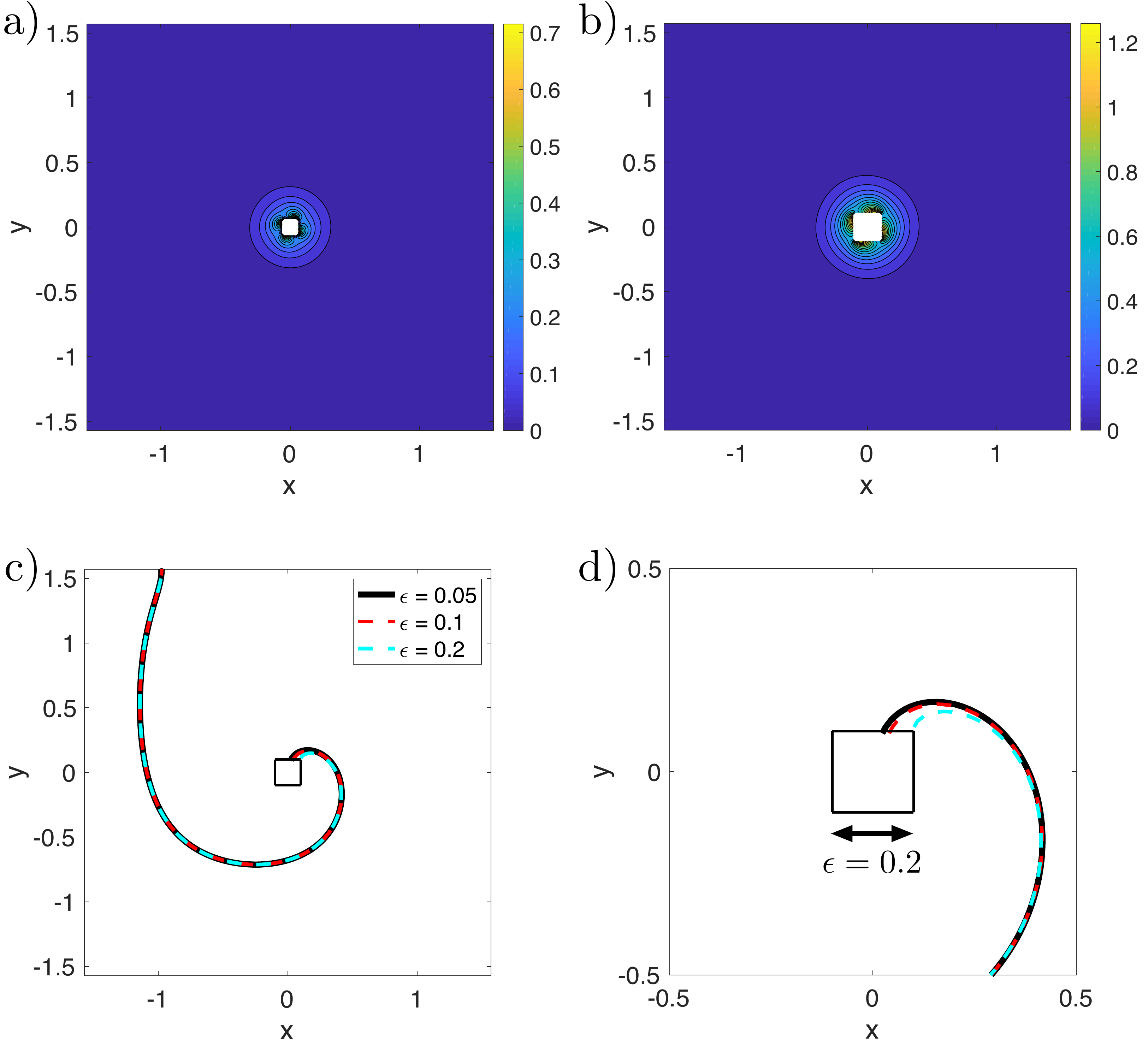} 
   \caption{Sensitivity of the MRT function $T$ to the size of the interior domain boundary $\epsilon$.  (a) Absolute difference $|T(x,y,\epsilon_1)-T(x,y,\epsilon_2)|$ at each point in grid for $\epsilon_1=0.05$ and $\epsilon_2=0.1$.  The inner square has side    length 0.1.   (b) Absolute difference $|T(x,y,\epsilon_1)-T(x,y,\epsilon_2)|$ for $\epsilon_1=0.05$ and $\epsilon_2=0.2$. The inner square has side length 0.2.  Mean period $\tbar=16.23$.  The relative shift in the time function due to changing the inner boundary size is less than one part in a thousand throughout most of the region.  Immediately adjacent to the inner boundary, relative deviations occur of up to 8\% of $\tbar$. (c) Comparison of a single MRT isochron (chosen to have a common value of $T$ at the outer boundary) for $\epsilon=0.05$ (solid black curve), $\epsilon=0.1$ (red dashed curve), and $\epsilon=0.2$ (cyan dashed curve).   
   The inner square has side length 0.2. (d) As in panel C, magnified to show detail.  The isochrons obtained with different interior domain sizes $\epsilon$ differ appreciably only in a neighborhood of size $O(\epsilon)$.}
   \label{fig:boundary_sensitivity}
\end{figure}

\section{A Further Example: Stuart-Landau oscillator with counterrotating  phase}
\label{ssec:counterrot}
We discuss one additional nontrivial example which was suggested by \rot{Newby and Schwemmer}, a Stuart-Landau oscillator with a true amplitude-phase dependence \rot{\cite{NewSch14}}. In polar coordinates, the model is given by
\begin{align}
\dot{\theta} &= \omega  +\gamma c(1-r^2)+ \sqrt{2D} \xi_1 (t),\quad 
\label{eq:counterrot}
\dot{r} = -\gamma r (r^2 - 1) + \sqrt{2D} \xi_2 (t).
\end{align}
On and inside the deterministic limit cycle ($r\le 1$), the deterministic field still induces on average a counterclockwise (mathematically positive) rotation, while for points sufficiently outside the limit cycle ($r^2>1+\omega/(\gamma c)$) the rotation is clockwise (cf. \bi{counterrotate-results}, left). The stochastic trajectory goes indeed both ways, depending on the value of the radius.
\begin{figure}[h!]
\centerline{
\includegraphics[width=\linewidth]{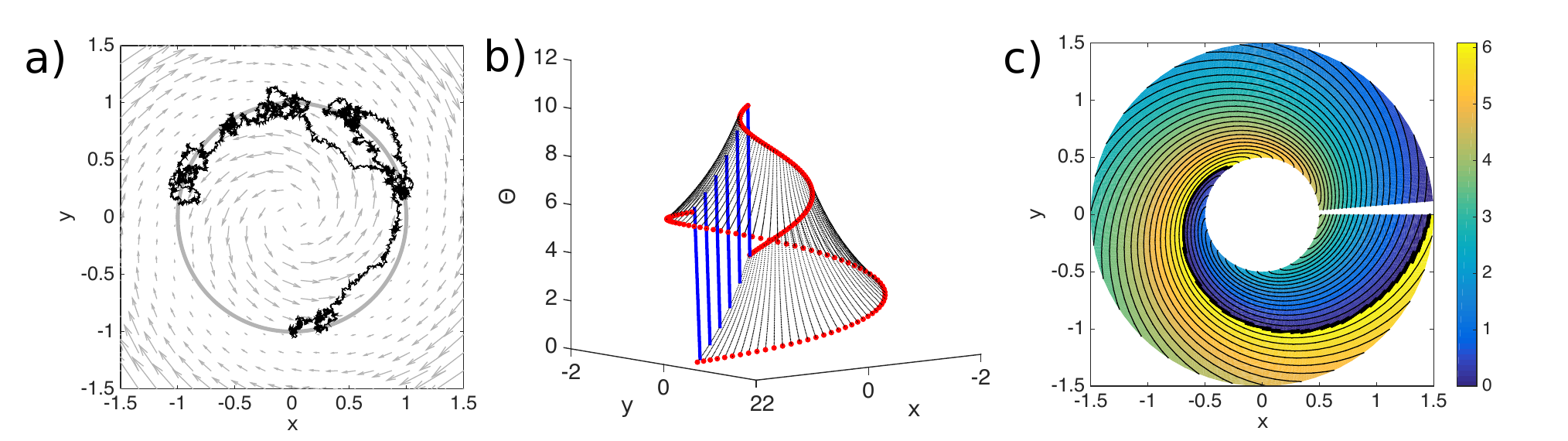} 
}
  \caption{Stuart-Landau oscillator with counterrotating phase,  \e{counterrot}, with $\omega=1$, $\gamma = 15$, $c=4$, and $D=0.18$. Sample trajectory ($\Delta t=10^{-4}$) with initial condition at $(0,-1)$ with the deterministic limit cycle shown by the bold grey line and the deterministic velocity field (a) [note that the field rotates clockwise for sufficiently large radius, which is the mentioned counterrotation], \rot{MRT phase $\Theta(x,y)$}  from the  finite-difference-scheme 
  solution (b) and its contour lines, \textit{i.e.}~the isochrons of the \rot{MRT} phase (c).   
\rot{For the parameters used, $\overline{T}=6.284$.}
}
  \label{fig:counterrotate-results}
\end{figure}

All the isochrons, here computed as the contour lines of the PDE solution, spiral inward, which can be understood as follows. Points with large radius have to get a head start because they go for a while in the wrong direction and need some time to turn (by a mean-driven or noise-induced reduction of the radius variable). Points on the inside of the circle, have a considerably higher rotation speed and have to be started with a certain delay,
at an earlier geometric phase. Of course, these arguments also apply to the deterministic system and we see that noise sources with moderate intensity do not change this picture qualitatively.

\section{Ito versus Stratonovich interpretation}
\label{supp:Stratonovich}

For completeness, we provide here the form of the PDE governing the mean--return-time function $T(x,y)$ when the underlying stochastic differential equations
\ba \label{eq:langevin_supp}
\dot{x}&=&f_x(x,y)+g_{x1}(x,y)\xi_1(t)+g_{x2}(x,y)\xi_2(t)\\
\nn
\dot{y}&=&f_y(x,y)+g_{y1}(x,y)\xi_1(t)+g_{y2}(x,y)\xi_2(t)
\ea
are understood in the Ito and the Stratonovich interpretations \cite{Gar97}.
In both cases the mean--return-time function $T(x,y)$ satisfies the inhomogeneous PDE
$\mathcal{L}^\dagger[T]=-1$ with adjoint reflecting boundary conditions on the inner and outer domain boundaries, and a jump of $\overline{T}$ across an arbitrary simple curve connecting the inner and outer boundaries.  The form of the adjoint Kolmogorov operator (generator of the Markov process) is given below.  

\subsection{Ito interpretation}

Under the Ito interpretation of the SDE \e{langevin_supp}, the form of $\mathcal{L}^\dagger$ is the same as that given with coordinates $(\alpha,\beta)$ in the main text (cf.~\eqref{eq:Ldag_u}).
\begin{align}\label{eq:Ldag_u_supp1}
{\cal L}^\dagger[u]=&\left[f_x\frac\partial{\partial x}+f_y\frac\partial{\partial y}+\frac12\left(\mathcal{G}_{xx}\frac{\partial^2}{\partial x^2}+\mathcal{G}_{xy}\frac{\partial^2}{\partial xy}+\mathcal{G}_{yx}\frac{\partial^2}{\partial yx}+\mathcal{G}_{yy}\frac{\partial^2}{\partial y^2}\right)\right][u],
\end{align}
where $\mathcal{G}=gg^\intercal$ and $g=\matrix{cc}{g_{x1}&g_{x2}\\g_{y1}&g_{y2}}$.  
The adjoint reflecting boundary conditions at an exterior boundary with normal vector $(n_1,n_2)^\intercal$ are 
\be
\sum_{j}n_j\sum_{k} \mathcal{G}_{jk}\partial_k T =0.
\ee

Writing out $\mathcal{L}^\dagger$ explicitly in terms of our drift and diffusion coefficients gives
\begin{align}\label{eq:Ldag_u_supp2}
{\cal L}^\dagger=&f_x\frac\partial{\partial x}+f_y\frac\partial{\partial y}
+\frac12
\left((g_{x1}^2+g_{x2}^2)\frac{\partial^2}{\partial x^2}
+2(g_{x1}g_{y1}+g_{x2}g_{y2})\frac{\partial^2}{\partial xy}
+(g_{y1}^2+g_{y2}^2)\frac{\partial^2}{\partial y^2}
\right).
\end{align}

\subsection{Stratonovich interpretation}

If we interpret the SDE \e{langevin_supp} in the sense of Stratonovich, this is equivalent to an Ito interpretation of an SDE with modified drift coefficients $\hat{f}_x,\hat{f}_y$:
\ba \label{eq:langevin_Strat_supp}
\dot{x}&=&\hat{f}_x(x,y)+g_{x1}(x,y)\xi_1(t)+g_{x2}(x,y)\xi_2(t)\\
\nn
\dot{y}&=&\hat{f}_y(x,y)+g_{y1}(x,y)\xi_1(t)+g_{y2}(x,y)\xi_2(t),
\ea
where 
\ba \label{eq:Strat_f_hat}
\hat{f}_x&= f_x+\frac12
\left( g_{x1}\frac{\partial}{\partial x} g_{x1}+g_{x2}\frac{\partial}{\partial x} g_{x2}+g_{y1}\frac{\partial}{\partial y} g_{x1}+g_{y2}\frac{\partial}{\partial y} g_{x2} \right)
\\ \nn
\hat{f}_y&= f_y+\frac12
\left( g_{x1}\frac{\partial}{\partial x} g_{y1}+g_{x2}\frac{\partial}{\partial x} g_{y2}+g_{y1}\frac{\partial}{\partial y} g_{y1}+g_{y2}\frac{\partial}{\partial y} g_{y2} \right);
\ea
compare eqs.~(4.3.40, 4.3.43) of  \cite{Gar97}.  

Substituting the modified drift coefficients $\hat{f}$ in to \eqref{eq:Ldag_u_supp1} or \eqref{eq:Ldag_u_supp2} gives the form of the PDE satisfied by $T$.  The change of interpretation from Ito to Stratonovich does not change the diffusion coefficients $g_{ij}$, and thus the adjoint reflecting boundary conditions are identical in both cases.


\end{document}